\documentclass[leqno]{article}
\usepackage[frenchb,english]{babel}
\usepackage[T1]{fontenc}
\usepackage[utf8]{inputenc}
\usepackage{amsmath}
\usepackage{amssymb}
\usepackage{amsfonts}
\usepackage{enumerate}
\usepackage{vmargin}
\usepackage[all]{xy}
\usepackage{mathrsfs}
\usepackage{mathtools}
\usepackage{lmodern}
\usepackage{comment}
\usepackage[colorlinks=true,linkcolor=blue]{hyperref}
\setmarginsrb{3cm}{3cm}{3.5cm}{3cm}{0cm}{0cm}{1.5cm}{3cm}

\newcommand{\B}{\ensuremath{\mathrm{B}}}
\newcommand{\C}{\ensuremath{\mathbb{C}}}

\newcommand{\E}{\ensuremath{\mathbb{E}}}

\let\H\relax 
\newcommand{\H}{\mathrm{H}} 

\let\L\relax 
\newcommand{\L}{\mathrm{L}}
\newcommand{\M}{\mathrm{M}}
\newcommand{\N}{\ensuremath{\mathcal{N}}}

\newcommand{\R}{\ensuremath{\mathbb{R}}}

\newcommand{\W}{\mathrm{W}}

\newcommand{\e}{\mathrm{e}}

\renewcommand{\d}{\mathop{}\mathopen{}\mathrm{d}} 
\newcommand{\vect}{\ensuremath{\mathop{\mathrm{span}\,}\nolimits}}

\let\i\relax 
\newcommand{\i}{\mathrm{i}}
\newcommand{\ovl}{\overline}
\newcommand{\Id}{\mathrm{Id}} 
\newcommand{\VN}{\mathrm{VN}} 
\newcommand{\la}{\langle}\newcommand{\ra}{\rangle}
\renewcommand{\leq}{\ensuremath{\leqslant}}
\renewcommand{\geq}{\ensuremath{\geqslant}}
\newcommand{\qed}{\hfill \vrule height6pt  width6pt depth0pt}
\newcommand{\norm}[1]{ \| #1  \|}
\newcommand{\bnorm}[1]{ \big\| #1  \big\|}

\newcommand{\xra}{\xrightarrow} 
\newcommand{\co}{\colon}

\newcommand{\ot}{\otimes}
\newcommand{\otvn}{\ovl\ot}

\newcommand{\BMO}{{\mathrm{BMO}}}
\newcommand{\scr}{\mathscr} 

\newcommand{\ov}{\overset}

\newcommand{\Mult}{\mathrm{Mult}}
\DeclareMathOperator{\tr}{Tr}
\newtheorem{thm}{Theorem}[section]
\newtheorem{defi}[thm]{Definition}
\newtheorem{prop}[thm]{Proposition}

\newtheorem{cor}[thm]{Corollary}
\newtheorem{lemma}[thm]{Lemma}

\newtheorem{example}[thm]{Example}

\newtheorem{remark}[thm]{Remark}

\newenvironment{proof}[1][]{\noindent {\it Proof #1} : }{\hbox{~}\qed
\smallskip
}

\numberwithin{equation}{section}

\usepackage{tocloft}
\usepackage[nottoc,notlot,notlof]{tocbibind}

\let\OLDthebibliography\thebibliography
\renewcommand\thebibliography[1]{
  \OLDthebibliography{#1}
  \setlength{\parskip}{0pt}
  \setlength{\itemsep}{0pt plus 0.3ex}
}

\begin{document}
\selectlanguage{english}
\title{\bfseries{Dilations of markovian semigroups of measurable Schur multipliers}}
\date{}
\author{\bfseries{C\'edric Arhancet}}

\maketitle

\begin{abstract}
Using probabilistic tools, we prove that any weak* continuous semigroup $(T_t)_{t \geq 0}$ of selfadjoint unital completely positive measurable Schur multipliers acting on the space $\B(\L^2(X))$ of bounded operators on the Hilbert space $\L^2(X)$, where $X$ is a suitable measure space, can be dilated by a weak* continuous group of Markov $*$-automorphisms on a bigger von Neumann algebra. We also construct a Markov dilation of these semigroups. Our results imply the boundedness of the McIntosh's $\H^\infty$ functional calculus of the generators of these semigroups on the associated Schatten spaces and some interpolation results connected to $\BMO$-spaces. We also give an answer to a question of Steen, Todorov and Turowska on completely positive continuous Schur multipliers.
\end{abstract}


\makeatletter
 \renewcommand{\@makefntext}[1]{#1}
 \makeatother

\tableofcontents

\section{Introduction}
\label{sec:Introduction}

\hspace{0.4cm} The study of dilations of operators is of central importance in operator theory and has a long tradition in functional analysis. Indeed, dilations are powerful tools which allow to reduce general studies of operators to more tractable ones. 

Suppose $1<p<\infty$. In the spirit of Sz.-Nagy's dilation theorem for contractions on Hilbert spaces, Fendler \cite{Fen97} proved a dilation result for any strongly continuous semigroup $(T_t)_{t \geq 0}$ of positive contractions on an $\L^p$-space $\L^p(\Omega)$. More precisely, this theorem says that there exists a measure space $\Omega'$, two positive contractions $J \co \L^p(\Omega) \to \L^p(\Omega')$ and $P \co \L^p(\Omega') \to \L^p(\Omega)$ and a strongly continuous group of positive invertible isometries $(U_t)_{t \in \R}$ on $\L^p(\Omega')$ such that 
\begin{equation}
\label{Fendler-dilation}
T_t = PU_tJ	
\end{equation}
for any $t \geq 0$, see also \cite{Fen98}. Note that in this situation the map $J \co \L^p(\Omega) \to \L^p(\Omega')$ is an isometric embedding whereas the map $JP \co \L^p(\Omega') \to \L^p(\Omega')$ is a contractive projection.

In the noncommutative setting, measure spaces and $\L^p$-spaces are replaced by von Neumann algebras and noncommutative $\L^p$-spaces and positive maps by completely positive maps. In their remarkable paper \cite{JLM07}, Junge and Le Merdy essentially\footnote{\thefootnote. The authors prove that there exists no ``reasonable'' analogue of a variant of Fendler's result for a discrete semigroup $(T^k)_{k \geq 0}$ of completely positive contractions.} showed that there is no hope to have a ``reasonable'' analog of Fendler's theorem for semigroups of completely positive contractions acting on noncommutative $\L^p$-spaces. It is a \textit{striking} difference with the world of classical (=commutative) $\L^p$-spaces of measure spaces.

Let $X$ be a $\sigma$-finite measure space. In this paper, our first main result (Theorem \ref{Th-dilation-continuous}) gives a dilation on weak* continuous semigroups of selfadjoint unital completely positive measurable Schur multipliers acting on $\B(\L^2(X))$ in the spirit of \eqref{Fendler-dilation} but at the level $p=\infty$.  Our construction heavily relies on the use of Gaussian processes. Our dilation induces an isometric dilation similar to the one of Fendler's theorem for the strongly continuous semigroup induced by the semigroup $(T_t)_{t \geq 0}$ on the Schatten space $S^p_X \ov{\mathrm{def}}{=} S^p(\L^2(X))$ for any $1 \leq p <\infty$. Finally, note that in this paper the weak* continuity of the semigroup means that the map $\R^+ \to \mathbb{C}$, $t \mapsto \la T_t(x),y \ra_{\B(\L^2(X)),S^1_X}$ is continuous for any $x \in \B(\L^2(X))$ and any $y \in S^1_X$. 

Moreover, this result is analogous to the result of \cite{Arh20} which provides a dilation of weak* continuous semigroups of selfadjoint unital completely positive Fourier multipliers acting on the group von Neumann algebra $\VN(G)$ of a locally compact group $G$. Note also that in the case where $X$ is a \textit{finite set} equipped with the counting measure, it is possible to use the considerably more complicated approach of \cite{Arh19} (relying on the use of ultrafilters) whose assumptions are satisfied by the results of \cite{Arh13a} and \cite{Ric08} for obtaining a dilation. But even in this case, our new and more general approach gives an explicit and a more useful dilation since the von Neumann algebra where the group $(U_t)_{t \in \R}$ acts is injective. This point is important for applications in vector-valued noncommutative $\L^p$-spaces since we need injective von Neumann algebras in this context, see \cite{Pis98}. We refer to the papers \cite{Arh13b}, \cite{ALM14}, \cite{AFM17} and \cite{HaM11} for related things.  Finally, note that Schur multipliers are crucial operators in noncommutative analysis and are connected to a considerable number of topics as harmonic analysis, double operator integrals, perturbation theory, Grothendieck's theorem. More recently, the semigroups of Schur mutipliers considered in this paper was connected to noncommutative geometry in the memoir \cite{ArK22a}.

One of the important consequences of Fendler's theorem is the boundedness, for the generator of a strongly continuous semigroup $(T_t)_{t \geq 0}$ of positive contractions, of a bounded $\H^\infty$ functional calculus which is a fundamental tool in various areas: harmonic analysis of semigroups, multiplier theory, Kato's square root problem, maximal regularity in parabolic equations and control theory. For detailed information, we refer the reader to \cite{Haa06}, \cite{JMX06}, \cite{KuW04}, to the survey \cite{LeM07} and to the recent book \cite{HvNVW18} and references therein. Our results also give a similar result on $\H^\infty$ functional calculus in the noncommutative context in the case of semigroups of measurable Schur multipliers. 

Recall that a locally integrable function $f \co \R^n \to \mathbb{C}$ is said to have bounded mean oscillation if
$$
\norm{f}_{\BMO}  
\ov{\mathrm{def}}{=} \sup_{Q \in \mathcal{Q}} \frac{1}{|Q|} \int_Q   \left| f(y) - f_Q \right|^2  \d y < \infty
$$
where $f_Q \ov{\mathrm{def}}{=} \frac{1}{|Q|} \int_Q f$ is the average of $f$ over $Q$, $|Q|$ the volume of $Q$ and $\mathcal{Q}$ is the set of all cubes in $\mathbb{R}^n$. The quotient space of bounded mean oscillation functions modulo the space of constant functions with the previous seminorm is a Banach space. The importance of the $\BMO$-norm and the $\BMO$-space lies in the fact that they arise as end-point spaces for estimating the bounds of linear maps on function spaces on $\mathbb{R}^n$. This includes many Calder\'on-Zygmund operators. By interpolation, $\BMO$-spaces is an effective tool to obtain $\L^p$-bounds of Fourier multipliers. In this sense, the $\BMO$-space is a natural substitute of the space $\L^\infty(\R^n)$.

Spaces of functions with bounded means oscillation can also be studied through semigroups. If $\mathcal{T}=(T_t)_{t \geq 0}$ is a Markov semigroup on an $\L^\infty$-space $\L^\infty(\Omega)$ of a (finite) measure space $\Omega$, then we can define a $\BMO$-seminorm by
$$
\norm{f}_{\BMO_{\mathcal{T}}} 
\ov{\mathrm{def}}{=} \sup_{t \geq 0} \bnorm{T_t\left(|f - T_t(f)|^2 \right) }_\infty^{\frac{1}{2}}. 
$$
More generally, Junge and Mei introduced in \cite{JM12} several noncommutative semigroup $\BMO$-spaces starting from a Markov semigroup on a semifinite von Neumann algebra (i.e. a noncommutative $\L^\infty$-space). Relations between these spaces are studied and interpolation results are obtained. A crucial ingredient of their approach is Markov dilations of semigroups that allow one to use martingale theory and derive results from this probabilistic setting. See also \cite{FMS19}, \cite{JiW17}, \cite{JM10} and \cite{JuZ15} for other applications of Markov dilations. 
    
The second main result of this paper is Theorem \ref{Th-Markov-dilation-Schur} which gives the construction of a Markov dilation of each weak* continuous semigroup of selfadjoint unital completely positive measurable Schur multipliers acting on $\B(\L^2(X))$. Again, our construction relies on probabilistic tools. From this result, we obtain interpolation results with $\BMO$-spaces as end-point for these semigroups, see Theorem \ref{Th-interpolation-BMO} and Remark \ref{Remarque-ultime}. Probably, our results will be useful for estimating norms of Schur multipliers on Schatten spaces. Finally, we equally refer to \cite{Cas19} and \cite{CJSZ20} for strongly related papers.

\paragraph{Structure of the paper} The paper is organized as follows. The next section \ref{sec:preliminaries} gives background. In Section \ref{Sec-complements}, we give useful observations on measurable positive definite kernels and negative definite kernels. In Section \ref{sec-description}, we give complements on measurable Schur multipliers. We answer a question of Steen, Todorov and Turowska in Remark \ref{Remark-answer}. We also give a precise description of weak* continuous semigroups of unital selfadjoint completely positive measurable Schur multipliers on some suitable measure spaces. Section \ref{sec-dilation-Schur} gives a proof of our main result of dilation of these semigroups. Next in Section \ref{sec-Markov}, we also construct a Markov dilation for these semigroups. In the last section \ref{sec-Applications}, we describe some applications of our results to functional calculus and interpolation.

\section{Preliminaries}
\label{sec:preliminaries}

In this paper, we suppose that the considered measures spaces are complete (i.e. every subset of every null set is measurable).

\paragraph{Isonormal processes} Let $H$ be a real Hilbert space. An $H$-isonormal process on a probability space $(\Omega,\mu)$ \cite[Definition 1.1.1]{Nua06} \cite[Definition 6.5]{Neer07} is a linear mapping $\W \co H \to \L^0(\Omega)$ from $H$ into the space $\L^0(\Omega)$ of measurable functions on $\Omega$ with the following properties:
\begin{flalign}
& \label{isonormal-gaussian} \text{for any $h \in H$ the random variable $\W(h)$ is a centered real Gaussian,} \\
&\label{esperance-isonormal} \text{for any } h_1, h_2 \in H \text{ we have } \E\big(\W(h_1) \W(h_2)\big)= \langle h_1, h_2\rangle_H, \\
&\label{density-isonormal} \text{the linear span of the products } \W(h_1)\W(h_2)\cdots \W(h_m), 
\text{ with } m \geq 0 \text{ and } h_1,\ldots, h_m \\
&\nonumber\text{in }H, \text{ is dense in the real Hilbert space $\L^2_\R(\Omega)$.}\end{flalign}  
Here, we make the convention that the empty product, corresponding to $m=0$ in \eqref{density-isonormal}, is the constant function $1$. Moreover, $\E$ is used to denote the expected value.

If $(e_i)_{i \in I}$ is an orthonormal basis of $H$ and if $(\gamma_i)_{i \in I}$ is a family of independent standard Gaussian random variables on a probability space $\Omega$ then for any $h \in H$, the family $(\gamma_i \langle h,e_i\rangle_H)_{i \in I}$ is summable in $\L^2(\Omega)$ and
\begin{equation}
\label{Concrete-W}
\W(h)
\ov{\mathrm{def}}{=} \sum_{i \in I} \gamma_i \langle h,e_i\rangle_H, \quad h \in H
\end{equation}
defines an $H$-isonormal process.

Recall that the span of elements $\e^{\i\W(h)}$ is weak* dense in $\L^\infty(\Omega)$ by \cite[Remark 2.15]{Jan97}.  
Using \cite[Proposition E.2.2]{HvNVW18} with $t$ instead of $\xi$ and by observing by \eqref{esperance-isonormal} that the variance $\E(\W(h)^2)$ of the Gaussian variable $\W(h)$ is equal to $\norm{h}_H^2$, we see that
\begin{equation}
\label{Esperance-exponentielle-complexe}
\E\big(\e^{\i t\W(h)}\big)
=\e^{-\frac{t^2}{2} \norm{h}_H^2}, \quad t \in \R, h \in H.
\end{equation}

If $u \co H \to H$ is a contraction, we denote by $\Gamma_\infty(u) \co \L^\infty(\Omega) \to \L^\infty(\Omega)$ the (symmetric) second quantization of $u$ acting on the \textit{complex} Banach space $\L^\infty(\Omega)$. Recall that the map $\Gamma_\infty(u) \co \L^\infty(\Omega) \to \L^\infty(\Omega)$ preserves the integral\footnote{\thefootnote. That means that for any $f \in \L^\infty(\Omega)$ we have $\int_{\Omega} \Gamma_\infty(u)f \d\mu=\int_{\Omega} f \d\mu$.}. If $u$ is a surjective isometry we have
\begin{equation}
\label{SQ1}
\Gamma_\infty(u) \big(\e^{\i\W(h)}\big)
=\e^{\i\W(u(h))}, \quad h \in H
\end{equation}
and $\Gamma_\infty(u) \co \L^\infty(\Omega) \to \L^\infty(\Omega)$ is a $*$-automorphism of the von Neumann algebra $\L^\infty(\Omega)$.

Furthermore, the second quantization functor $\Gamma$ satisfies the following elementary result \cite[Lemma 2.1]{Arh20}. In the part 1, we suppose that the construction\footnote{\thefootnote. The existence of a proof of Lemma \ref{Lemma-semigroup-continuous} without \eqref{Concrete-W} is unclear.} is given by the concrete representation \eqref{Concrete-W}.   

\begin{lemma}
\label{Lemma-semigroup-continuous} 
\begin{enumerate}
	\item If $\L^\infty(\Omega)$ is equipped with the weak* topology then the map $H \to \L^\infty(\Omega)$, $h \mapsto \e^{\i\W(h)}$ is continuous.
	\item If $(U_t)_{t \in \R}$ is a strongly continuous group of orthogonal operators acting on the Hilbert space $H$ then $(\Gamma_\infty(U_t))_{t \in \R}$ is a weak* continuous group of operators acting on the Banach space $\L^\infty(\Omega)$.
\end{enumerate}
\end{lemma}

Let $H$ be a real Hilbert space. Following \cite[Definition 2.2]{NVW15} and \cite[Definition 6.11]{Neer07}, we say that an $\L^2_\R(\R^+,H)$-isonormal process $\W$ is an $H$-cylindrical Brownian motion. In this case, for any $t \geq 0$ and any $h \in H$, we let
\begin{equation}
\label{Def-Wt-h}
\W_t(h)\ov{\mathrm{def}}{=}\W\big(1_{[0,t]} \ot h\big).	
\end{equation}
We introduce the filtration $(\scr{F}_t)_{t \geq 0}$ defined by 
\begin{equation}
\label{Filtration-H-cylindrical}
\scr{F}_t\ov{\mathrm{def}}{=} \sigma\big(\W_r(h) : r \in [0,t], h \in H\big),
\end{equation} 
that is the $\sigma$-algebra generated by the random variables $\W_r(h)$ for $r \in [0,t]$ and $h \in H$.

By \cite[p. 77]{Neer07}, for any fixed $h \in H$, the family $(\W_t(h))_{t \geq 0}$ is a Brownian motion. This means that \cite[Definition 6.2]{Neer07}
\begin{flalign}
& \label{isonormal-almost} \text{$\W_0(h) = 0$ almost surely,} \\
&\label{difference-1} \text{$\W_t(h)-\W_s(h)$ is Gaussian with variance $(t-s)\norm{h}_H^2$ for any $0 \leq s \leq t$,} \\
&\label{difference-2} \text{$\W_t(h)-\W_s(h)$ is independent of $\{\W_r(h) :  r \in [0 ,s] \}$ for any $0 \leq s \leq t$}.  
\end{flalign}
Indeed by \cite[p. 163]{Neer07}, 
\begin{flalign}
\label{difference-3}
&\text{the increment $\W_t(h)-\W_s(h)$ is independent of the $\sigma$-algebra $\scr{F}_s$}.
\end{flalign}
Moreover, by \cite[p. 163]{Neer07} the family $(\W_t(h))_{t \geq 0}$ is a martingale with respect to $(\scr{F}_t)_{t \geq 0}$. In particular, the random variable $\W_t(h)$ is $\scr{F}_t$-measurable. If $0 \leq s \leq t$, note that 
$$
\norm{1_{]s,t]} \ot h}_{\L^2_\R(\R^+,H)}^2
= \bnorm{1_{]s,t]}}_{\L^2_\R(\R^+)}^2 \norm{h}_H^2
=(t-s) \norm{h}_H^2.
$$ 
Using \eqref{Esperance-exponentielle-complexe} together with the previous computation, we obtain
\begin{equation}
\label{Esperance-exponentielle-complexe-2}
\mathrm{E}\big(\e^{\i \W(1_{]s,t]} \ot h)}\big)
=\e^{-\frac{t-s}{2} \norm{h}_H^2}, \quad 0 \leq s \leq t,\ h \in H.
\end{equation}

\paragraph{Probabilities}
%
Let $\Omega$ be a probability space and let $X$ be a Banach space. If $f \in \L^1(\Omega,X)$ is independent of the sub-$\sigma$-algebra $\scr{F}$, then by \cite[Proposition 2.6.35]{HvNVW16} its conditional expectation $\E(f|\scr{F})$ with respect to $\scr{F}$ is given by the constant function:
\begin{equation}
\label{Hy-2.6.35}
\E(f|\scr{F}) 
= \E(f).
\end{equation}

\paragraph{Hilbert-Schmidt operators} Let $X$ be a $\sigma$-finite measure space. We will use the space $S^\infty_X \ov{\mathrm{def}}{=} S^\infty(\L^2(X))$ of compact operators, its dual $S^1_X$ and the space $\B(\L^2(X))$ of bounded operators on the Hilbert space $\L^2(X)$. If $f \in \L^2(X \times X)$, we denote the associated Hilbert-Schmidt operator by
\begin{equation}
\label{Def-de-Kf}
\begin{array}{cccc}
  K_f  \co &  \L^2(X)   &  \longrightarrow   & \L^2(X)   \\
    &   \xi  &  \longmapsto       &  \int_{X} f(\cdot,y)\xi(y) \d y  \\
\end{array}.	
\end{equation}
We will also use the notation $K((x,y) \mapsto f(x,y))$ for $K_f$. With the notation $\check{f}(x,y) \ov{\mathrm{def}}{=} f(y,x)$, we have 
\begin{equation}
\label{Adjoint}
(K_f)^*
=K_{\check{\ovl{f}}}.
\end{equation}
The operator $K_f$ is selfadjoint if and only if $f(x,y)=\ovl{f(y,x)}$ almost everywhere. 

For any $f,g \in \L^2(X \times X)$, we have $K_fK_g=K_h$ where $h \in \L^2(X \times X)$ is defined by 
\begin{equation}
\label{composition-Hilbert-Schmidt}
h(x,y)
=\int_X f(x,z)g(z,y) \d z
\end{equation}
with $\norm{h}_{2} \leq \norm{f}_{2} \norm{g}_{2}$.


\paragraph{Kernels}
Let $X$ be a set. A function $\varphi \co X \times X \to \C$ is called hermitian if $\varphi(y,x) = \ovl{\varphi(x,y)}$ for any $x,y \in X$. A function $\varphi \co X \times X \to \mathbb{C}$ is a positive definite kernel if for any integer $n \geq 1$, any $c_1,\ldots,c_n \in \mathbb{C}$ and any $x_1,\ldots,x_n \in X$ we have
\begin{equation}
\label{Positive-definite-kernel}
\sum_{i,j=1}^n c_i\ovl{c_j} \varphi(x_i,x_j) 
\geq 0.
\end{equation}
In this definition, it is enough to consider mutually different elements $x_1,\ldots,x_n$ of $X$. By \cite[1.5 p. 68]{BCR84}, any positive definite kernel is hermitian.

\begin{example} \normalfont
\label{Example-kernel-3}
If $X$ is a finite set then the function $\varphi \co X \times X \to \mathbb{C}$ is a positive definite kernel if and only if the matrix $[\varphi(x_i,x_j)]$ is positive definite.
\end{example}

\begin{example} \normalfont
\label{Example-kernel-2}
Let $H$ be a Hilbert space and let $\alpha \co X \to H$ be a map. By \cite[p. 9]{Boz87}, the function $\varphi \co X \times X \to \mathbb{C}$ defined by
\begin{equation}
\label{Def-varphi}
\varphi(x,y) 
= \la \alpha_x,\alpha_y\ra_H, \quad x,y \in X
\end{equation}
is a positive definite kernel. 
\end{example}

Conversely, by \cite[p. 82]{BCR84} every positive definite kernel $\varphi$ is of this form for some suitable Hilbert space $H$ and some map $\alpha \co X \to H$. Moreover, if $\varphi$ is real-valued, we can take a real Hilbert space.

We say that $\varphi \co X \times X \to \C$ is a negative definite kernel if $\varphi$ is hermitian and if for any integer $n \geq 1$, any $c_1,\ldots,c_n \in \mathbb{C}$ with $c_1+\cdots +c_n=0$ and any $x_1,\ldots,x_n \in X$ we have
$$
\sum_{i,j=1}^n c_i \ovl{c_j} \varphi(x_i,x_j) 
\leq 0.
$$

\begin{example} \normalfont
\label{Example-kernel}
Let $H$ be a real Hilbert space and let $\alpha \co X \to H$ be a map. The map $\psi \co X \times X \to \R$ defined by $\psi(x,y) = \norm{\alpha_x - \alpha_y}_H^2$ is a negative definite kernel. 
\end{example}

By \cite[Proposition 3.2 p. 82]{BCR84}
, every \textit{real-valued} negative definite kernel $\psi$ which vanishes on the diagonal $\{(x,x) : x \in X\}$ is of this form for a real Hilbert space $H$.

The following is the famous Schoenberg's theorem \cite[Theorem 2.2 p. 74]{BCR84} which relates negative definite kernels and positive definite kernels.

\begin{thm}
\label{Th-Schoenberg}
Let $X$ be a set and let $\psi \co X \times X \to \mathbb{C}$ be a map. Then $\psi$ is a negative definite kernel if and only if for any $t >0$ the map $\e^{-t\psi} \co X \times X \to \mathbb{C}$ is a positive definite kernel.
\end{thm}


\section{Complements on measurable kernels}
\label{Sec-complements}

The following lemma is folklore. For the sake of completeness, we give an argument following Cartier \cite[Lemme p. 287]{Car81}. 

\begin{lemma}
\label{Lemma-weak-measurable}
Let $H$ be a Hilbert space. Suppose that $X$ is a set equipped with a $\sigma$-algebra $\mathcal{A}$. Let $\alpha \co X \to H$ be a function. For any $x,y \in X$, we define $\varphi \co X \times X \to \mathbb{C}$ by \eqref{Def-varphi}. If $\varphi$ is measurable with respect to $\mathcal{A} \ot \mathcal{A}$ then the map $\alpha$ is weak measurable, i.e. for any $h \in H$, the map $X \to \mathbb{C}$, $x \mapsto \langle \alpha_x,h \rangle_H$ is measurable.
\end{lemma}

\begin{proof}
We denote by $H_1$ the set of vectors $h$ of $H$ such that the function $X \to \mathbb{C}$, $x \mapsto \la \alpha_x ,h \ra_H$ is measurable. It is obvious that $H_1$ is a closed subspace of $H$. It is clear that any vector of $H$ which is orthogonal to the range of $\alpha$ belongs to $H_1$, i.e. $\alpha(X)^\perp \subset H_1$. We infer that $\ovl{\vect \alpha(X)}^\perp \subset H_1$. Since the partial functions of $\varphi$ are measurable, for any $y \in X$ the map $X \to \mathbb{C}$, $x \mapsto \langle \alpha_x,\alpha_{y} \rangle_H$ is measurable. So we also have the inclusion $\alpha(X) \subset H_1$, hence $\ovl{\vect \alpha(X)} \subset H_1$. We conclude that $H \subset H_1$. 
\end{proof}

\begin{lemma}
\label{Lemma-2-measurable}
Suppose that $X$ is a set equipped with a $\sigma$-algebra $\mathcal{A}$. Let $\psi \co X \times X \to \R$ be a measurable real-valued negative definite kernel which vanishes on the diagonal $\{ (x,x) : x \in X \}$. Then there exists a real Hilbert space $H$ and a weak measurable map $\alpha \co X \to H$ such that
$$
\psi(x,y) 
= \norm{\alpha_x - \alpha_y}_H^2, \quad x,y \in X.
$$
\end{lemma}

\begin{proof}
Fix $x_0 \in X$. Recall that by \cite[Lemma 2.1 p. 74]{BCR84} the function $\varphi \co X \times X \to \mathbb{R}$ defined by 
\begin{equation}
\label{Eq-inter-infty}
\varphi(x,y)
=\frac{1}{2}\big[\psi(x,x_0)+\psi(x_0,y)-\psi(x,y)\big], \quad x,y \in X
\end{equation}  
is a positive definite kernel. Since $\psi$ is measurable, this kernel is measurable. So there exists a real Hilbert space $H$ and a function $\alpha \co X \to H$ satisfying \eqref{Def-varphi} which is weak measurable by Lemma \ref{Lemma-weak-measurable}. Now, for any $x,y \in X$, we have
\begin{align*}
\MoveEqLeft
\norm{\alpha_x - \alpha_y}_H^2            
=\norm{\alpha_x}_H^2+\norm{\alpha_y}_H^2-2 \la \alpha_x,\alpha_y \ra_H
\ov{\eqref{Def-varphi}}{=} \varphi(x,x)+\varphi(y,y)-2\varphi(x,y) \\
&\ov{\eqref{Eq-inter-infty}}{=}\frac{1}{2}\big(\psi(x,x_0)+\psi(x_0,x)-\psi(x,x)\big)+\frac{1}{2}\big(\psi(y,x_0)+\psi(x_0,y)-\psi(y,y)\big)\\
&-\big(\psi(x,x_0)+\psi(x_0,y)-\psi(x,y)\big)
=\psi(x,y).
\end{align*}
\end{proof}

%

%

We introduce the following definition in the spirit of \cite[p. 25]{Mer21}.

\begin{defi}
\label{Def-int-pos-def}
Let $X$ be a measure space. We say that a measurable map $\varphi \co X \times X \to \mathbb{C}$ is an integrally positive definite kernel if for any $\xi \in \L^1(X)$ the function $(x,y) \mapsto \varphi(x,y) \ovl{\xi(x)} \xi(y)$ is integrable on $X \times X$ and 
\begin{equation}
\label{Cartier-int-pos}
\iint_{X \times X} \varphi(x,y) \ovl{\xi(x)} \xi(y) \d x \d y 
\geq 0.
\end{equation}
\end{defi}

Let $\varphi \co X \times X \to \mathbb{C}$ be a measurable positive definite kernel on a $\sigma$-finite measure space $X$. Cartier showed in \cite[p. 302]{Car81} that \eqref{Cartier-int-pos} is satisfied for any measurable function $\xi \co X \to \mathbb{C}$ if the previous integral is absolutely convergent. Moreover, Cartier proved the following result \cite[Corollaire p. 301]{Car81}.

\begin{prop}
\label{Prop-Magic}
Suppose that $X$ is a $\sigma$-finite measure space. Let $\varphi \co X \times X \to \mathbb{C}$ be a measurable and bounded function. Then $\varphi$ is equal almost everywhere to a bounded and measurable positive definite kernel on $X \times X$ if and only if $\varphi$ is an integrally positive definite kernel. 
\end{prop}

\begin{remark} \normalfont
The author naively thought that he proved the ``if'' part of a $\L^2(X)$-variant of this result in a previous version of this paper relying on the version of Mercer's theorem for finite measure spaces stated in \cite[Theorem 3.a.1 p. 145]{Kon86}. Unfortunately, the proof of \cite[Theorem 3.a.1 p. 145]{Kon86} has a large gap\footnote{\thefootnote. It was confirmed by email by the author of the book.}. In particular, the estimate $\sup_n \norm{f_n}_\infty <\infty$ is not proved and seems doubtful and consequently the arguments of convergence of \cite{Kon86} are inexact.  
\end{remark}

%

\begin{remark} \normalfont
It would be interesting to study variants of Definition \ref{Def-int-pos-def} by replacing the space $\L^1(X)$ by another space of functions in the spirit of \cite{HuW75} and references therein. 

\end{remark}

The following is \cite[Th\'eor\`eme p. 283]{Car81}\footnote{\thefootnote. The ``only if'' part is true without the assumption ``with support $X$''.} and would be used in the next section.

\begin{prop}
\label{Prop-Cartier-2}
Let $X$ be a locally compact space equipped with a Radon measure with support $X$ and $\varphi \co X \times X \to \mathbb{C}$ be a continuous function. Then $\varphi$ is a positive definite kernel if and only if for any continuous function $\xi \co X \to \mathbb{C}$ with compact support we have \eqref{Cartier-int-pos}.
\end{prop}

\section{A description of Markov semigroups of measurable Schur multipliers}
\label{sec-description}

Let $X$ be a $\sigma$-finite measure space. We say that a function $\varphi \in \L^\infty(X \times X)$ induces a measurable Schur multiplier on $\B(\L^2(X))$ if the map $S^2_X \mapsto \B(\L^2(X))$, $K_{f} \mapsto K_{\varphi f}$ induces a bounded operator from $S^\infty_X$ into $\B(\L^2(X))$. In this case, the operator $S^\infty_X \mapsto \B(\L^2(X))$, $K_{f}\mapsto K_{\varphi f}$ admits by \cite[Lemma A.2.2 p. 360]{BLM04} a unique weak* extension $M_\varphi \co \B(\L^2(X)) \to \B(\L^2(X))$ called the measurable Schur multiplier associated with $\varphi$. It is known that in this case
\begin{equation}
\label{ine-infty}
\norm{\varphi}_{\L^\infty(X \times X)} 
\leq \norm{M_\varphi}_{\B(\L^2(X)) \to \B(\L^2(X))}.
\end{equation}
We refer to the surveys \cite{ToT10} and \cite{Tod15} for more information. See also \cite{Spr04}.

\begin{example} \normalfont
If the set $X=\{1,\ldots,n\}$ is equipped with the counting measure, we can identify the space $\B(\L^2(X))$ with the matrix algebra $\M_n$. Then each operator $K_{f}$ identifies to the matrix $[f(i,j)]$. A Schur multiplier is given by a map $M_\varphi \co \M_n \to \M_n$, $[f(i,j)] \mapsto [\varphi(i,j)f(i,j)]$. 
\end{example}

We say that a measurable Schur multiplier $M_\varphi \co \B(\L^2(X)) \to \B(\L^2(X))$ is selfadjoint if its restriction $S^2_X \mapsto S^2_X$, $K_{f} \mapsto K_{\varphi f}$ on the Hilbert space $S^2_X$ is selfadjoint. See \cite[Section 5]{JMX06} for more information on a related notion for operators acting on von Neumann algebras.

\begin{prop}
\label{Prop-selfadjoint}
Suppose that $X$ is a $\sigma$-finite measure space. A measurable Schur multiplier $M_\varphi \co \B(\L^2(X)) \to \B(\L^2(X))$ is selfadjoint if and only if the function $\varphi$ is real-valued almost everywhere.
\end{prop}

\begin{proof}
The selfadjointness property is equivalent to
$$
\tr\big(M_\varphi(K_f)K_g^*\big)
=\tr\big(K_f(M_\varphi(K_g))^*\big), \quad \text{i.e.} \quad \tr\big(K_{\varphi f}K_g^*\big)
=\tr\big(K_f(K_{\varphi g})^*\big)
$$
for any $f,g \in \L^2(X \times X)$. Since $\L^2(X \times X) \to S^2_X$, $f \mapsto K_f$ is an isometry from the Hilbert space $\L^2(X \times X)$ onto the Hilbert space $S^2_X$ of Hilbert-Schmidt operators on $\L^2(X)$, that means that
$$
\iint_{X \times X} \varphi(x,y)f(x,y)\ovl{g(x,y)} \d x\d y
=\iint_{X \times X} f(x,y) \ovl{\varphi(x,y)g(x,y)} \d x\d y.
$$
Note that $f\ovl{g}$ belongs to $\L^1(X \times X)$ and that each function of $\L^1(X \times X)$ has this form. By duality, we deduce that $\varphi=\ovl{\varphi}$ almost everywhere.
\end{proof}

We will show that a measurable Schur multiplier $M_\varphi \co \B(\L^2(X)) \to \B(\L^2(X))$ is unital if and only if its restriction on $S^1_X$ is trace preserving. We need some information on Arens products to prove some preliminary result. Let $A$ be a Banach algebra. Recall that we can equip its bidual $A^{**}$ with two natural products, the left and right Arens poducts defined in \cite[pp. 78-79]{BLM04} on $A^{**}$. We say $A$ is Arens regular if these products coincides on $A^{**}$. Note that the algebra $S^\infty_X$ of compact operators is Arens regular by \cite[Theorem 9.1.39 p. 838]{Pal01}.

\begin{prop}
Let $X$ be a $\sigma$-finite measure space. The map $S^2_X \to S^2_X$, $K_f \mapsto K_{\check{f}}$ extends to an involutive normal $*$-antiautomorphism $R \co \B(\L^2(X)) \to \B(\L^2(X))$.
\end{prop}

\begin{proof} 
For any $f \in \L^2(X \times X)$, we let $R(K_f) \ov{\mathrm{def}}{=} K_{\check{f}}$. On the one hand, for any $f,g \in \L^2(X \times X)$, the Hilbert-Schmidt operator $K_{f}K_g \co \L^2(X) \to \L^2(X)$ is associated to the function $h$ defined by \eqref{composition-Hilbert-Schmidt} and  for almost all $(x,y) \in X \times X$ we have 
$$
\check{h}(x,y)= \int_X f(y,z)g(z,x) \d z.
$$ 
On the other hand, by \eqref{composition-Hilbert-Schmidt} the Hilbert-Schmidt operator $K_{\check{g}}K_{\check{f}} \co \L^2(X) \to \L^2(X)$ is associated to the function 
$$
(x,y) \mapsto \int_X \check{g}(x,z)\check{f}(z,y) \d z=\int_X f(y,z) g(z,x) \d z.
$$ 
We obtain that $R(K_fK_g)=R(K_g)R(K_f)$. Moreover, we have 
$$
(R(K_{f}))^*
=\big(K_{\check{f}}\big)^*
\ov{\eqref{Adjoint}}{=} K_{\ovl{f}}
=R(K_{\check{\ovl{f}}})
\ov{\eqref{Adjoint}}{=} R((K_f)^*).
$$
We conclude that the map $S^2_X \to S^2_X$, $K_f \mapsto K_{\check{f}}$ is an involutive $*$-antiautomorphism. From the formula $(K_f)^*=K_{\check{\ovl{f}}}$, it is not difficult to see that the \textit{Banach} adjoint of $K_f$ is $K_{\check{f}}$. Hence $\norm{K_f}_{\B(\L^2(X))}=\norm{K_{\check{f}}}_{\B(\L^2(X))}$. We deduce an involutive $*$-antiautomorphism $R \co S^\infty_X \to S^\infty_X$. By \cite[Lemma A.2.2 p. 360]{BLM04}, we deduce a unique weak* continuous extension $R \co (S^\infty_X)^{**} \to \B(\L^2(X))$. Moreover, by \cite[2.5.5 p. 79]{BLM04} this map is a $*$-antiautomorphism where the double dual $(S^\infty_X)^{**}$ is equipped with the Arens product. Finally, note that by \cite[Theorem 9.1.39 p. 838]{Pal01} the algebra $(S^\infty_X)^{**}$ identifies with the algebra $\B(\L^2(X))$.
\end{proof}

We introduce the following duality bracket
\begin{equation}
\label{Duality-bracket}
\langle z,y \rangle_{\B(\L^2(X)), S^1_X}
\ov{\mathrm{def}}{=} \tr(R(z)y), \quad z \in \B(\L^2(X)), y \in S^1_X.
\end{equation}

\begin{prop}
Let $X$ be a $\sigma$-finite measure space. The preadjoint $(M_{\varphi})_* \co S^1_X\to S^1_X$ of a measurable Schur multiplier $M_{\varphi} \co \B(\L^2(X)) \to \B(\L^2(X))$ coincides with its restriction on $S^1_X$, i.e. for any $z \in \B(\L^2(X))$ and $y \in S^1_X$ we have
\begin{equation}
\label{auto-adjoint}
\big\langle M_{\varphi}(z),y \big\rangle_{\B(\L^2(X)), S^1_X}
=\big\langle z,M_{\varphi}(y) \big\rangle_{\B(\L^2(X)), S^1_X}.
\end{equation}
\end{prop}

\begin{proof}
Recall that the map $\L^2(X \times X) \to S^2_X$, $f \mapsto K_f$ is an isometry from the Hilbert space $\L^2(X \times X)$ onto the Hilbert space $S^2_X$ of Hilbert-Schmidt operators on $\L^2(X)$. For any $K_f,K_g \in S^1_X$, we deduce that
\begin{align*}
\MoveEqLeft
\big\langle M_{\varphi}(K_f),K_g \big\rangle_{\B(\L^2(X)), S^1_X}
=\big\langle K_{\varphi f},K_g \big\rangle_{\B(\L^2(X)), S^1_X} 
\ov{\eqref{Duality-bracket}}{=} \tr(R(K_{\varphi f})K_g) 
=\tr(K_{\check{\varphi} \check{f}}K_g) \\
&=\iint_{X \times X} \varphi(y,x)f(y,x) g(y,x) \d x\d y.
\end{align*} 
Moreover, we have
\begin{align*}
\MoveEqLeft
\big\langle K_f,K_{\varphi g} \big\rangle_{\B(\L^2(X)), S^1_X}            
\ov{\eqref{Duality-bracket}}{=} \tr(R(K_f) K_{\varphi g})=\tr(K_{\check{f}} K_{\varphi g}) 
=\iint_{X \times X} f(y,x)\varphi(y,x) g(y,x) \d x\d y.
\end{align*} 
We conclude by density.
\end{proof}

Now, we prove the following equivalence by a standard argument. 

\begin{prop}
\label{Prop-trace-unital}
Suppose that $X$ is a $\sigma$-finite measure space. A measurable Schur multiplier $M_\varphi \co \B(\L^2(X)) \to \B(\L^2(X))$ is unital if and only if it is trace preserving when restricted to the space $S^1_X$.
\end{prop}

\begin{proof}
It is elementary to check that an $*$-antiautomorphism is necessarily unital. If the measurable Schur multiplier $M_\varphi \co \B(\L^2(X)) \to \B(\L^2(X))$ is unital then for any $y \in S^1_X$ we have $\big\langle M_{\varphi}(1),y \big\rangle_{\B(\L^2(X)), S^1_X}
\ov{\eqref{auto-adjoint}}{=} \big\langle 1,M_{\varphi}(y) \big\rangle_{\B(\L^2(X)), S^1_X}$, 
that is $\tr y=\tr M_{\varphi}(y)$. We conclude that $M_\varphi$ is trace preserving.

If the measurable Schur multiplier $M_\varphi \co \B(\L^2(X)) \to \B(\L^2(X))$ is trace preserving then for any $x \in S^1_X$ we have
\begin{align*}
\MoveEqLeft
\tr (R(M_\varphi(1))x)
\ov{\eqref{Duality-bracket}}{=} \big\langle M_{\varphi}(1),x \big\rangle_{\B(\L^2(X)), S^1_X} 
\ov{\eqref{auto-adjoint}}{=} \big\langle 1,M_{\varphi}(x) \big\rangle_{\B(\L^2(X)), S^1_X} \\
&\ov{\eqref{Duality-bracket}}{=} \tr(R(1)M_\varphi(x))
=\tr(R(1)x).
\end{align*}
We conclude by duality that $M_\varphi(1)=1$.
\end{proof}

It is proved in \cite[Theorem 3.5]{Duq22} that a measurable Schur multiplier $M_\varphi \co \B(\L^2(X)) \to \B(\L^2(X))$ is unital if and only if $\varphi(x,x)=1$ almost everywhere on $X$. We deduce the following result.

\begin{cor}
\label{cor-trace}
Suppose that $X$ is a $\sigma$-finite measure space. Let $M_\varphi \co \B(\L^2(X)) \to \B(\L^2(X))$ be a measurable Schur multiplier. The following conditions are equivalent.
\begin{enumerate}
	\item $M_\varphi$ is trace preserving when restricted to the space $S^1_X$.
	\item $M_\varphi$ is unital.
	\item $\varphi(x,x)=1$ almost everywhere on $X$.
\end{enumerate}
\end{cor}

In the following statement, we say that $\varphi$ induces a (completely) positive measurable Schur multiplier if the linear map $M_\varphi \co \B(\L^2(X)) \to \B(\L^2(X))$ is (completely) positive.

\begin{prop}
\label{Prop-Carac-cp-Schur multiplier}
Let $X$ be a $\sigma$-finite measure space. Let $\varphi \in \L^\infty(X \times X)$. Consider the following properties.
\begin{enumerate}
	\item $\varphi$ induces a completely positive measurable Schur multiplier
	
	\item $\varphi$ induces a positive measurable Schur multiplier

	\item the function $\varphi$ is an integrally positive definite kernel.
\end{enumerate}
We have $1 \iff 2 \Rightarrow 3$. 
\end{prop}

\begin{proof}
1. $\Rightarrow$ 2. It is obvious.

2. $\Rightarrow$ 1. That is essentially a standard application of \cite[Lemma 4.3]{STT14}. See also \cite[Exercise 4.15]{Tod14}.

2. $\Rightarrow$ 3. Let $\xi \in \L^1(X)$. We can write $\xi=\xi_1\xi_2$ for some functions $\xi_1,\xi_2 \in \L^2(X)$. Note that the operator $K_{\ovl{\xi_1} \ot \xi_1} \co \L^2(X) \to \L^2(X)$ is positive. Indeed,  for any $\eta \in \L^2(X)$ we have 
\begin{align*}
\MoveEqLeft
\big\la K_{\ovl{\xi_1} \ot \xi_1}(\eta),\eta \big\ra_{\L^2(X)}       
\ov{\eqref{Def-de-Kf}}{=} \bigg(\int_X \xi_1 \eta \bigg)\big\la \ovl{\xi_1} ,\eta \big\ra_{\L^2(X)} \\
&=\ovl{\big\la \ovl{\xi_1} ,\eta \big\ra_{\L^2(X)}} \big\la \ovl{\xi_1} ,\eta \big\ra_{\L^2(X)}
=\big|\big\la \ovl{\xi_1} ,\eta \big\ra_{\L^2(X)} \big|^2
\geq 0.
\end{align*}
We deduce that $\la M_\varphi(K_{\ovl{\xi_1} \ot \xi_1})\xi_2,\xi_2 \ra_{\L^2(X)} \geq 0$. Finally, it suffices to observe that
\begin{align*}
\MoveEqLeft
\big\la M_\varphi(K_{\ovl{\xi_1} \ot \xi_1})\xi_2,\xi_2 \big\ra_{\L^2(X)}         
=\Big\la K_{\varphi(\ovl{\xi_1} \ot \xi_1)} \xi_2,\xi_2 \Big\ra_{\L^2(X)} 
=\int_X \big(K_{\varphi(\ovl{\xi_1} \ot \xi_1)} \xi_2\big)(x)\ovl{\xi_2(x)} \d x \\
&\ov{\eqref{Def-de-Kf}}{=} \int_X \bigg(\int_X \varphi(x,y)\ovl{\xi_1(x)}  \xi_1(y)\xi_2(y)\d y\bigg)\ovl{\xi_2(x)} \d x \\
&=\iint_{X \times X} \varphi(x,y) \ovl{\xi_1(x)} \xi_1(y)\xi_2(y) \ovl{\xi_2(x)} \d x \d y 
=\iint_{X \times X} \varphi(x,y) \ovl{\xi(x)} \xi(y) \d x \d y. 
\end{align*}
So \eqref{Cartier-int-pos} is satisfied.
\end{proof}

\begin{remark} \normalfont
Let $X$ be a locally compact space equipped with a Radon measure with full support. Let $\varphi \in \L^\infty(X \times X)$ be a function which induces a measurable Schur multiplier $M_\varphi \co \B(\L^2(X)) \to \B(\L^2(X))$. The implication ``3'' implies 1'' is true. Suppose that $\varphi$ is an integrally positive definite kernel. By Proposition \ref{Prop-Magic}, we can suppose that $\varphi$ is equal to a bounded and measurable positive definite kernel on $X \times X$. Let $K_f \co \L^2(X) \to \L^2(X)$ be a positive Hilbert-Schmidt operator with a \textit{continuous} symbol $f \co X \times X \to \mathbb{C}$. Then the inequality 
$$
\iint_{X \times X} f(x,y) \ovl{\xi(x)} \xi(y) \d x \d y 
\geq 0
$$ 
is clearly satisfied for any $\xi \in \L^2(X)$, in particular for any function $\xi \in \mathrm{C}_c(X)$. By Proposition \ref{Prop-Cartier-2}, we deduce that $f$ is the class of a positive definite kernel. Using the stability under pointwise product of the set of definite positive kernels \cite[Theorem 1.12 p. 69]{BCR84}, we see that the function $\varphi f$ of $\L^2(X \times X)$ is also the class of a positive definite kernel. By the observation following Definition \ref{Def-int-pos-def}, we deduce that for any $\xi \in \L^2(X)$ we have
$$
\iint_{X \times X} \varphi(x,y)f(x,y) \ovl{\xi(x)} \xi(y) \d x \d y 
\geq 0.
$$ 
Consequently $M_\varphi(K_f)=K_{\varphi f}$ is a positive operator on the Hilbert space $\L^2(X)$. We conclude with Kaplansky's theorem using the cone of positive Hilbert-Schmidt operators with continuous kernels. 
\end{remark}

\begin{remark} \normalfont
\label{Remark-answer}
Let $X$ be a locally compact space equipped with a Radon measure with full support. Here, we explain how to give an answer to \cite[Question 4.13 (ii)]{STT14}. Consider a positive continuous Schur multiplier $M_\varphi \co \B(\L^2(X)) \to \B(\L^2(X))$, i.e. a measurable Schur multiplier with a \textit{continuous} symbol $\varphi \co X \times X \to \mathbb{C}$. By Proposition \ref{Prop-Carac-cp-Schur multiplier}, the map $\varphi$ is an integrally positive definite kernel. In particular, for any continuous function $ \xi \co X \to \mathbb{C}$ with compact support, we have \eqref{Cartier-int-pos}. By Proposition \ref{Prop-Cartier-2}, we infer that the continuous function $\varphi$ is a positive definite kernel. Consequently, by \cite[p. 82]{BCR84} there exists a Hilbert space $H$ and some map $\alpha \co X \to H$ satisfying \eqref{Def-varphi}. Since the function $\varphi$ is continuous, we conclude by the (classical) result \cite[Lemme p. 279]{Car81}\footnote{\thefootnote. If \eqref{Def-varphi} is satisfied on a topological space $X$ then $\varphi$ is continuous if and only if $\alpha$ is continuous.} that the map $\alpha \co X \to H$ is necessarily continuous (here the Hilbert space $H$ is equipped with its usual topology). Since $\varphi \in \L^\infty(X \times X)$, the function $x \mapsto \varphi(x,x)=\la \alpha_x,\alpha_x \ra_H$ is bounded. We conclude that $\alpha$ is bounded. 

With sharp contrast, we think that the  answer to the similar question \cite[Question 4.13 (i)]{STT14} for (not necessarily positive) continuous Schur multipliers is negative. 
\end{remark}


The following is a generalization of \cite[Proposition 5.4]{Arh13a}. 
Recall that in this paper the weak* continuity of a semigroup $(T_t)_{t \geq 0}$ means that the map $\R^+ \to \mathbb{C}$, $t \mapsto \la T_t(x),y \ra_{\B(\L^2(X)),S^1_X}$ is continuous for any $x \in \B(\L^2(X))$ and any $y \in S^1_X$.

\begin{prop}
\label{Prop-description-Schur-mult-cp}
\begin{enumerate}
	\item Let $X$ be a second countable Radon measure space. If $(T_t)_{t \geq 0}$ is a weak* continuous semigroup of selfadjoint unital completely positive measurable Schur multipliers on $\B(\L^2(X))$ with continuous symbols then there exists a real Hilbert space $H$ and a continuous function $\alpha \co X \to H$ such that the symbol of $T_t$ is
\begin{equation}
\label{symbol-semigroup-Schur}
\phi_t(x,y)
=\e^{-t\norm{\alpha_x-\alpha_y}_{H}^2}, \quad \text{for almost all }x,y \in X.	
\end{equation}	
\item Let $X$ be a (complete) $\sigma$-finite measure space such that  If $(T_t)_{t \geq 0}$ is a weak* continuous semigroup of selfadjoint unital completely positive measurable Schur multipliers on $\B(\L^2(X))$ then there exists a real Hilbert space $H$ and a weak measurable function $\alpha \co X \to H$ such that the symbol of $T_t$ is given by \eqref{symbol-semigroup-Schur}.
\end{enumerate}
\end{prop}

\begin{proof}
1. For any $t \geq 0$, let $\phi_t \co X \times X \to \C$ be the continuous symbol of the Schur multiplier $T_t$. Since the map $T_t$ is selfadjoint, the function $\phi_t$ is real-valued by Proposition \ref{Prop-selfadjoint}. By Corollary \ref{cor-trace}, we have $\phi_t(x,x)=1$ for any $x \in X$ and any $t \geq 0$. By Proposition \ref{Prop-Carac-cp-Schur multiplier}, each $\phi_t$ is an integrally positive definite kernel. We deduce by Proposition \ref{Prop-Cartier-2} that each $\phi_t$ is a positive definite kernel. Finally, we have $\phi_0=1$. 

For any $t,t' \geq 0$ and any $f \in \L^2(X \times X)$, the relation $T_tT_{t'}(K_{f})=T_{t+t'}(K_f)$ gives the equality $K_{\phi_{t}\phi_{t'} f}=K_{\phi_{t+t'}f}$. We infer that $\phi_t\phi_{t'}f=\phi_{t+t'}f$ in the space $\L^2(X \times X)$. It is apparent that $(\phi_t)_{t \geq 0}$ defines a semigroup $(\Mult_t)_{t \geq 0}$ of multiplication operators acting on $\L^2(X \times X)$ defined by
\begin{equation}
\label{Def-de-Mult}
\Mult_t(f)
\ov{\mathrm{def}}{=} \phi_tf, \quad t \geq 0, f \in \L^2(X \times X).
\end{equation}
Note that \eqref{ine-infty} implies that $(\Mult_t)_{t \geq 0}$ is a contraction semigroup. We will show the weak continuity of this semigroup. For any functions $f,g \in \L^2(X \times X)$ such that $K_g \in S^1_X$, we have using the weak* continuity of the semigroup $(T_t)_{t \geq 0}$
\begin{align*}
\MoveEqLeft
\big\la \Mult_t(f),g \big\ra_{\L^2(X \times X)}         
\ov{\eqref{Def-de-Mult}}{=} \la \phi_tf,g \ra_{\L^2(X \times X)} 
=\big\la K_{\phi_t f}, K_g \big\ra_{S^2_X} \\
&=\big\la T_{t}(K_f), K_g \big\ra_{\B(\L^2(X)), S^1_X} 
\xra[t \to 0]{} \la K_f, K_g \ra_{\B(\L^2(X)), S^1_X} 
=\la f,g \ra_{\L^2(X \times X)}
\end{align*}
(here, we make a slight abuse of notation for the definition of the brackets). Note that the set of functions $g \in \L^2(X \times X)$ such that $K_g \in S^1_X$ is dense in $\L^2(X \times X)$ since $S^1_X$ is dense in the space $S^2_X$. Now, with \cite[2.71 (a) p. 234]{Meg98}, it is clear (using the contractivity of the semigroup) that this semigroup is weak continuous. By \cite[Theorem 5.8 p. 40]{EnN00}, this semigroup is even strongly continuous. By \cite[Proposition 4.12 p. 32]{EnN00} there exists a measurable function $\psi \co X \times X \to \R$ such that for almost all $x,y \in X$
$$
\phi_t(x,y)
=\e^{-t\psi(x,y)}, \quad t \geq 0.
$$ 
We deduce that for almost all $x,y \in X$ we have if $t>0$
$$
\psi(x,y)
=-\frac{1}{t}\log \phi_t(x,y).
$$
Since $\phi_t$ is continuous, we can suppose that $\psi$ is continuous. Moreover, for any $x \in X$ we have 
$$
\psi(x,x)
=-\frac{1}{t}\log \phi_t(x,x)
=0.
$$ 
By Schoenberg's Theorem (Theorem \ref{Th-Schoenberg}), the function $\psi$ defines a real-valued negative definite kernel which vanishes on the diagonal of $X \times X$. Finally, we use the characterization of continuous real-valued negative definite kernels of \cite[Theorem C.2.3]{BHV08}.

2. For any $t \geq 0$, let $\phi_t \co X \times X \to \C$ be the symbol of $T_t$. Since the map $T_t$ is selfadjoint, the function $\phi_t$ is real-valued almost everywhere by Proposition \ref{Prop-selfadjoint}. Consequently we can suppose that each function $\phi_t$ is real-valued. By Proposition \ref{Prop-Carac-cp-Schur multiplier}, each $\phi_t$ is an  integrally positive definite kernel. Consequently, we can suppose by Proposition \ref{Prop-Magic} that the function $\phi_t$ is a real-valued positive negative kernel and bounded (by 1).  

By Corollary \ref{cor-trace}, we have $\phi_t(x,x)=1$ for almost all $x \in X$ and any $t \geq 0$. For any $t>0$, let $A_t$ be the subset of $X$ such that $\phi_t(x,x) \not=1$. The subset $\{ (x,x) : x \in A_t \}$ is a subset of $A_t \times A_t$, hence negligible. By modifying $\phi_t$ on $A_t$, we can suppose that $\phi_t(x,x) =1$ for any $x \in X$. Indeed, it is clear by observing the definition \eqref{Positive-definite-kernel} that the property \textit{positive definite} is not lost by this change.

The end of the proof is similar to the proof of the first point replacing \cite[Theorem C.2.3]{BHV08} by Lemma \ref{Lemma-2-measurable}.
\end{proof}

\begin{remark} \normalfont
In the opposite direction, consider a \textit{separable} real Hilbert space $H$ and a weak measurable function $\alpha \co X \to H$. Let $(e_n)_{n \geq 1}$ be an orhonormal basis of $H$. We can write 
\begin{equation}
\label{countable-sum}
\la \alpha_x,\alpha_y \ra_H
=\sum_{n \geq 1} \la \alpha_x, e_n \ra_H \la e_n, \alpha_y \ra_H, \quad x,y \in X.
\end{equation} 
So the map $X \times X \to \mathbb{C}$, $(x,y) \mapsto \la \alpha_x,\alpha_y \ra_H$ is measurable with respect to the $\sigma$-algebra $\mathcal{A} \ot \mathcal{A}$. Since we have
$$
\norm{\alpha_x - \alpha_y}_H^2            
=\norm{\alpha_x}_H^2+\norm{\alpha_y}_H^2-2 \la \alpha_x,\alpha_y \ra_H,
$$
we see that the symbol $\phi_t$ defined in \eqref{symbol-semigroup-Schur} is measurable.

Note that $\phi_t$ converges to the function 1 when $t \to 0$ for the weak* topology of $\L^\infty(X \times X)$. Indeed, for any $f \in \L^1(X \times X)$ we have using dominated convergence theorem
\begin{align*}
\MoveEqLeft
\la\phi_t,f \ra_{\L^\infty(X \times X),\L^1(X \times X)}         
=\int_{X \times X} \e^{-t\norm{\alpha_x-\alpha_y}_{H}^2} f(x,y)\d x \d y 
\xra[t \to 0]{} \int_{X \times X}  f(x,y)\d x \d y \\
&=\la 1 ,f \ra_{\L^\infty(X \times X),\L^1(X \times X)}.
\end{align*}
With this convergence, it is left to the reader to show with \cite[Lemma 4.25]{ArK22b} that the symbols \eqref{symbol-semigroup-Schur} define a weak* continuous semigroup $(T_t)_{t \geq 0}$.

If $H$ is not separable, we can obtain the measurability of the symbols \eqref{symbol-semigroup-Schur} if we make the stronger assertion that $\alpha$ is strongly measurable \cite[Definition 1.1.14 p. 8]{HvNVW16}, which is equivalent by Pettis measurability theorem \cite[Theorem 1.1.20 p. 10]{HvNVW16} to say that $\alpha$ is weak measurable and essentially separably valued, i.e. there exists a negligible subset $N$ of $X$ such that $\alpha(X-N)$ is separable. In this case, we can find a separable Hilbert space $H$ containing $\alpha(X-N)$. Consequently, with an orhonormal basis $(e_n)_{n \geq 1}$ of $H$ \eqref{countable-sum} is true on $(X-N) \times (X-N)$. We conclude as previously (if $X$ is complete).
\end{remark}

\section{Dilations of semigroups of measurable Schur multipliers}
\label{sec-dilation-Schur}

Let $X$ be a $\sigma$-finite measure space. If $g \in \L^\infty(X)$, we will use the multiplication operator $\Mult_{g} \co \L^2(X) \to \L^2(X)$, $\xi \mapsto g\xi$. Recall that $(\Mult_{g})^*=\Mult_{\ovl{g}}$. We will use the following elementary result which describes the composition of a Hilbert-Schmidt operator $K_f \co \L^2(X) \to \L^2(X)$ with multiplication operators. Recall that the space of Hilbert-Schmidt operators is an ideal of the algebra $\B(\L^2(X))$ of bounded operators on the Hilbert space $\L^2(X)$.

\begin{lemma}
\label{Useful-lemma}
For any $\xi \in \L^2(X)$, any $g,h \in \L^\infty(X)$ and any $f \in \L^2(X \times X)$ we have for almost all $x \in X$
\begin{equation}
\label{formule-calcul-mult}
\big((\Mult_{g} K_f \Mult_{h})(\xi)\big)(x)
=\int_{X} g(x)h(y)f(x,y)\xi(x) \d y.
\end{equation}
So the Hilbert-Schmidt operator $\Mult_{g} K_f \Mult_{h} \co \L^2(X) \to \L^2(X)$ is associated with the function $(x,y) \mapsto g(x)h(y)f(x,y)\xi(x)$.
\end{lemma}

\begin{proof}
For almost all $x \in X$, we have
\begin{align*}
\MoveEqLeft
\big((\Mult_{g} K_f \Mult_{h})(\xi)\big)(x)           
		=\big((\Mult_{g} K_f)(h\xi)\big)(x) 
		\ov{\eqref{Def-de-Kf}}{=} \bigg(\Mult_{g}\bigg(\int_{X} f(\cdot,y)h(y)\xi(y) \d y\bigg)\bigg)(x)\\
		&=\bigg(g(\cdot)\int_{X} f(\cdot,y)h(y)\xi(y) \d y\bigg)(x)
		=\int_{X} g(x)h(y)f(x,y)\xi(x) \d x.
\end{align*} 
\end{proof}

%

The following is the first main result of this paper. 

\begin{thm}
\label{Th-dilation-continuous}
Let $X$ be a $\sigma$-finite measure space. Let $H$ be a real Hilbert space and $\alpha \co X \to H$ be a strongly measurable map. Then the symbols \eqref{symbol-semigroup-Schur} for $t \geq 0$ defines a weak* continuous semigroup $(T_t)_{t \geq 0}$ of selfadjoint unital completely positive Schur multipliers on $\B(\L^2(X))$. Moreover, there exists an injective von Neumann algebra $M$ equipped with a normal semifinite faithful trace, a weak* continuous group $(U_t)_{t \in \R}$ of trace preserving $*$-automorphisms of $M$, a unital trace preserving injective normal $*$-homomorphism $J \co \B(\L^2(X)) \to M$ such that
\begin{equation}
\label{Dilation-Tt}
T_t
=\E U_t J	
\end{equation}
for any $t \geq 0$, where $\E \co M \to \B(\L^2(X))$ is the canonical trace preserving normal faithful conditional expectation associated with $J$.
\end{thm}

\begin{proof}
Let $\W \co \L^2_\R(\R,H) \to \L^0(\Omega)$ be an $\L^2_\R(\R,H)$-isonormal process on a probability space $(\Omega,\mu)$ as in \eqref{Concrete-W}. We define the von Neumann algebra $M \ov{\mathrm{def}}{=} \L^\infty(\Omega) \otvn \B(\L^2(X))$. Note that $M$ is an injective von Neumann algebra. We equip this von Neumann algebra with the normal semifinite faithful trace $\tau_M \ov{\mathrm{def}}{=} \int_{\Omega} \cdot \ot \tr$. By \cite[Theorem 1.22.13]{Sak71}, we have a $*$-isomorphism $M=\L^\infty(\Omega,\B(\L^2(X)))$. We define the canonical unital normal injective $*$-homomorphism
\begin{equation}
\label{Def-de-J-Schur}
\begin{array}{cccc}
  J \co  &  \B(\L^2(X))   &  \longrightarrow   &   \L^\infty(\Omega) \otvn \B(\L^2(X)) \\
      &       z        &  \longmapsto       &  1 \ot z  \\
\end{array}.	
\end{equation}
It is clear that the map $J$ preserves the traces. We denote by $\E \co \L^\infty(\Omega) \otvn \B(\L^2(X)) \to \B(\L^2(X))$ the canonical trace preserving normal faithful conditional expectation of $M$ onto $\B(\L^2(X))$. Note that $\E=\int_\Omega \cdot \ot \Id_{\B(\L^2(X))}$. For almost all $\omega \in \Omega$ and any $t \geq 0$, we introduce with \eqref{Def-Wt-h} 
\begin{equation}
\label{Def-Kt-omega}
k_{t,\omega}(x)	
\ov{\mathrm{def}}{=} \e^{\sqrt{2}\i (\W_t(\alpha_x))(\omega)}, \quad x \in X.
\end{equation}
Let $(f_i)$ be an orthonormal basis of the separable Hilbert space $\L^2_\R(\R^+)$. By Pettis measurability theorem \cite[Theorem 1.1.20 p. 10]{HvNVW16}, there exists a negligible subset $N$ of $X$ such that $\alpha(X-N)$ is subset of a separable closed subspace $H_0$ of $H$. Let $(e_j)$ be an orthonormal basis of $H_0$. For any $t>0$ fixed, we can write almost everywhere (with \cite[Corollary 6.4.4 pp. 35-36]{HvNVW18})
\begin{equation}
(\W_t(\alpha_x))(\omega)
\ov{\eqref{Concrete-W}\eqref{Def-Wt-h}}{=} \sum_{i,j \in I} \gamma_{i,j}(\omega) \big\langle 1_{[0,t]} \ot \alpha_x,f_i \ot e_j \big\rangle_{\L^2_\R(\R^+,H)}
=\sum_{i,j \in I} \gamma_{i,j}(\omega) \big\langle 1_{[0,t]},f_i \big\rangle_{} \langle \alpha_x, e_j \rangle_H.
\end{equation}
We conclude that the function $X \times \Omega \to \mathbb{C}$, $(x,\omega) \mapsto k_{t,\omega}(x)$ is measurable.

If $t<0$, we define $k_{t,\omega}(x)$ by the same formula but with $\W_t(\alpha_x) \ov{\mathrm{def}}{=} \W(-1_{[t,0]} \ot \alpha_x)$. For any $t \in \R$, we can define an element $V_t$ of $\L^\infty(\Omega,\B(\L^2(X)))$ by 
\begin{equation}
\label{Def-Vt}
V_t(\omega)
\ov{\mathrm{def}}{=} \Mult_{k_{t,\omega}}.	
\end{equation}
Indeed, the function $V_t$ is weak* measurable since if we consider a normal linear form $f \co \B(\L^2(X)) \to \mathbb{C}$, $z \mapsto \sum_{n \geq 1} \la z(\xi_n), \eta_n \ra_{\L^2(X)}$ with $\sum_{n \geq 1} (\norm{\xi_n}_{\L^2(X)}+\norm{\eta_n}_{\L^2(X)}) < \infty$ \cite[Proposition 1.2.2 p. 9]{Li92}, we see that
$$
f\big(\Mult_{k_{t,\omega}}\big)
=\sum_{n \geq 1} \big\la \Mult_{k_{t,\omega}}\xi_n, \eta_n \big\ra_{\L^2(X)}
=\sum_{n \geq 1} \int_X k_{t,\omega} \xi_n \eta_n 
$$
is measurable since each integral defines a measurable function $\omega \to \int_X k_{t,\omega} \xi_n \eta_n$ by Fubini's theorem. Note that $V_t$ is an unitary element of $\L^\infty(\Omega,\B(\L^2(X)))$. Now, for any $t \in \R$ we define the linear map
\begin{equation}
\label{Def-Ut-Schur}
\begin{array}{cccc}
 \mathcal{U}_t \co   &  \L^\infty\big(\Omega,\B(\L^2(X))\big)   &  \longrightarrow   &  \L^\infty\big(\Omega,\B(\L^2(X))\big)  \\
        &  g   &  \longmapsto       &  V_t gV_t^* \\
\end{array}.	
\end{equation}
It is obvious that each map $\mathcal{U}_t$ is a trace preserving $*$-automorphism of the von Neumann algebra $M$. Moreover, for any elements $\xi,\eta$ of the Hilbert space $\L^2(\Omega,\L^2(X))=\L^2(\Omega \times X)$, we obtain using\footnote{\thefootnote. We have $1_{[0,t]} \xra[t \to t_0]{} 1_{[0,t_0]}$ and $1_{[t,0]} \xra[t \to t_0]{} 1_{[t_0,0]}$ in $\L^2_\R(\R)$.} the first part of Lemma \ref{Lemma-semigroup-continuous} since $\ovl{\xi} \eta \in \L^1(\Omega \times X)$
\begin{align*}
\MoveEqLeft
\big\langle V_t(\xi), \eta \big\rangle_{\L^2(\Omega,\L^2(X))}           
		=\int_{\Omega \times X} \ovl{V_t(\xi)(\omega,x)} \eta(\omega,x) \d\mu(\omega) \d x \\
		&\ov{\eqref{Def-Vt}}{=} \int_{\Omega \times X} \ovl{k_{t,\omega}(x)\xi(\omega,x)} \eta(\omega,x) \d\mu(\omega) \d x 
		\ov{\eqref{Def-Kt-omega}}{=} \int_{\Omega \times X} \e^{-\sqrt{2}\i (\W_t(\alpha_x))(\omega)} \ovl{\xi(\omega,x)} \eta(\omega,x) \d\mu(\omega) \d x
		\\
		&\xra[t \to 0]{} \int_{\Omega \times X} \e^{-\sqrt{2}\i (\W_{t_0}(\alpha_x))(\omega)}\ovl{\xi(\omega,x)} \eta(\omega,x)\d\mu(\omega)\d x
		=\langle V_{t_0}(\xi), \eta \rangle_{\L^2(\Omega,\L^2(X))}.	
\end{align*} 
So $(V_t)_{t \in \R}$ is a weak continuous family of unitaries hence a strongly continuous continuous family since by \cite[p. 41]{Str81} the weak operator topology and the strong operator topology coincide on the unitary group. By composition on bounded subsets, the map $\R \to \L^\infty(\Omega,\B(\L^2(X)))$, $t \mapsto \mathcal{U}_t(g)$ is strong operator continuous for any $g \in \L^\infty(\Omega,\B(\L^2(X)))$ and clearly even strong* operator continuous, hence weak* continuous. We conclude that $(\mathcal{U}_t)_{t \in \R}$ is a point-weak* continuous family of $*$-automorphisms. For any $f \in \L^2(X \times X)$ and any $t \in \R$, we have 
\begin{align}
\MoveEqLeft
\label{prem-series-2}
\mathcal{U}_t(1 \ot K_f)   
\ov{\eqref{Def-Ut-Schur}}{=} \omega \mapsto (V_t (1 \ot K_f) V_t^*)(\omega)
=\omega \mapsto V_t(\omega) K_f V_t(\omega)^* \\
&\ov{\eqref{Def-Vt}}{=} \omega \mapsto \Mult_{k_{t,\omega}} K_f (\Mult_{k_{t,\omega}})^* 
\nonumber 
\ov{\eqref{formule-calcul-mult}}{=} \omega \mapsto K\big((x,y) \mapsto k_{t,\omega}(x)\ovl{k_{t,\omega}(y)}f(x,y) \big) \nonumber \\
&\ov{\eqref{Def-Kt-omega}}{=} \omega \mapsto K\big((x,y) \mapsto \e^{\sqrt{2}\i (\W_t(\alpha_x-\alpha_y))(\omega)} f(x,y) \big).
\nonumber
\end{align} 
In particular, for almost all $\omega \in \Omega$, any $\xi \in \L^2(X)$ and almost all $y \in X$, we have
\begin{align}
\MoveEqLeft
\label{prem-series}
\big(\big[(\mathcal{U}_t(1 \ot K_f))(\omega)\big](\xi)\big)(y)         
\ov{\eqref{Def-de-Kf}}{=} \int_X \e^{\sqrt{2}\i (\W_t(\alpha_x-\alpha_y))(\omega)}	f(x,y) \xi(x) \d x. 
\end{align} 
For any $t \in \R$, we introduce the right shift $\mathcal{S}_t \co \L^2_\R(\R) \to \L^2_\R(\R)$. We consider the operator $S_t \ov{\mathrm{def}}{=} \Gamma_\infty(\mathcal{S}_t)  \ot \Id_{\B(\L^2(X))} \co \L^\infty(\Omega,\B(\L^2(X))) \to \L^\infty(\Omega,\B(\L^2(X)))$. Essentially, by the second part of Lemma \ref{Lemma-semigroup-continuous}, $(S_t)_{t \in \R}$ is a weak* continuous semigroup of $*$-automorphisms. We have
\begin{equation}
\label{Def-shift-2}
S_t\big(\e^{\i \W(1_{[s,s'[} \ot h)} \ot z\big)
\ov{\eqref{SQ1}}{=} \e^{\i \W(1_{[s+t,s'+t[} \ot h)} \ot z, \quad t \in \R, h \in H, z \in \B(\L^2(X)), s<s'.
\end{equation}
For any $t \in \R$, we define the $*$-automorphism $U_t \ov{\mathrm{def}}{=} \mathcal{U}_t S_t$. We will prove that $(U_t)_{t \in \R}$ is a group of operators. On the one hand, for any $t,t' \in \R$, we have
\begin{align*}
\MoveEqLeft
U_{t'}U_t\big(\e^{\i\W(1_{[s,s']} \ot h)} \ot 1\big)            
=\mathcal{U}_{t'} S_{t'} \mathcal{U}_t S_t\big(\e^{\i\W(1_{[s,s']} \ot h)} \ot 1\big) 
\ov{\eqref{Def-shift-2}}{=} \mathcal{U}_{t'} S_{t'} \mathcal{U}_t\big(\e^{\i\W(1_{[s+t,s'+t]} \ot h)} \ot 1\big) \\
&=\mathcal{U}_{t'} S_{t'} \big(\e^{\i\W(1_{[s+t,s'+t]} \ot h)} \ot 1\big) 
\ov{\eqref{Def-shift-2}}{=} \mathcal{U}_{t'}\big(\e^{\i\W(1_{[s+t+t',s'+t+t']} \ot h)} \ot 1\big) \\
&=\e^{\i\W(1_{[s+t+t',s'+t+t']} \ot h)} \ot 1 
=U_{t+t'}\big(\e^{\i\W(1_{[s,s']} \ot h)} \ot 1\big).
\end{align*}  
On the other hand, for any $t,t' \geq 0$ and any $\xi \in \L^2(X)$ we have
\begin{align*}
\MoveEqLeft
U_{t'}U_t(1 \ot K_f)         
=\mathcal{U}_{t'} S_{t'}\mathcal{U}_tS_t(1 \ot K_f) 
\ov{\eqref{Def-shift-2}}{=} \mathcal{U}_{t'} S_{t'}\mathcal{U}_t (1 \ot K_f) \\
&\ov{\eqref{prem-series-2}}{=} \mathcal{U}_{t'} S_{t'} \big(\omega \mapsto K\big((x,y) \mapsto \e^{\sqrt{2}\i (\W_t(\alpha_x-\alpha_y))(\omega)}	f(x,y)\big)\big)\\
&=\mathcal{U}_{t'} \big(\omega \mapsto K\big((x,y) \mapsto \e^{\sqrt{2}\i (\W(1_{[t',t+t']} \ot (\alpha_x-\alpha_y)))(\omega)}	f(x,y)\big)\big)\\
&\ov{\eqref{Def-Ut-Schur} \eqref{Def-Vt}}{=} \omega \mapsto \Mult_{k_{t',\omega}} K\big((x,y) \mapsto \e^{\sqrt{2}\i (\W(1_{[t',t+t']} \ot (\alpha_x-\alpha_y)))(\omega)}	f(x,y)\big) (\Mult_{k_{t',\omega}})^*\\
&\ov{\eqref{formule-calcul-mult}}{=} \omega \mapsto K\big((x,y) \mapsto k_{t',\omega}(x)\ovl{k_{t',\omega}}(y)\e^{\sqrt{2}\i (\W(1_{[t',t+t']} \ot (\alpha_x-\alpha_y)))(\omega)}	f(x,y) \big) \\
&\ov{\eqref{Def-Kt-omega}}{=} \omega \mapsto  K\big((x,y) \mapsto\e^{\sqrt{2}\i (\W_{t'}(\alpha_x-\alpha_y))(\omega)}\e^{\sqrt{2}\i (\W(1_{[t',t+t']} \ot (\alpha_x-\alpha_y)))(\omega)}	f(x,y) \big)\\
&= \omega \mapsto K\big((x,y) \mapsto \e^{\sqrt{2}\i (\W_{t+t'}(\alpha_x-\alpha_y))(\omega)} f(x,y) \big)
\ov{\eqref{prem-series}}{=} \mathcal{U}_{t+t'}(1 \ot K_f) 
=U_{t+t'}(1 \ot K_f).
\end{align*}  
Similarly for any $t \leq 0$ and any $t' \geq 0$ such that $t' \leq -t$ and any $\xi \in \L^2(X)$ we have
\begin{align*}
\MoveEqLeft
U_{t'}U_t(1 \ot K_f)         
=\mathcal{U}_{t'} S_{t'}\mathcal{U}_tS_t(1 \ot K_f) 
\ov{\eqref{Def-shift-2}}{=} \mathcal{U}_{t'} S_{t'}\mathcal{U}_t (1 \ot K_f) \\
&\ov{\eqref{prem-series-2}}{=} \mathcal{U}_{t'} S_{t'} \big(\omega \mapsto K\big((x,y) \mapsto \e^{\sqrt{2}\i (\W(-1_{[t,0]} \ot (\alpha_x-\alpha_y)))(\omega)}	f(x,y)\big)\big)\\
&=\mathcal{U}_{t'} \big(\omega \mapsto K\big((x,y) \mapsto \e^{\sqrt{2}\i (\W(-1_{[t'+t,t']} \ot (\alpha_x-\alpha_y)))(\omega)}	f(x,y)\big)\big)\\
&\ov{\eqref{Def-Ut-Schur} \eqref{Def-Vt}}{=} \omega \mapsto \Mult_{k_{t',\omega}} K\big((x,y) \mapsto \e^{\sqrt{2}\i (\W(-1_{[t'+t,t']} \ot (\alpha_x-\alpha_y)))(\omega)}	f(x,y)\big) (\Mult_{k_{t',\omega}})^*\\
&\ov{\eqref{formule-calcul-mult}}{=} \omega \mapsto K\big((x,y) \mapsto k_{t',\omega}(x)\ovl{k_{t',\omega}}(y)\e^{\sqrt{2}\i (\W(-1_{[t'+t,t']} \ot (\alpha_x-\alpha_y)))(\omega)}	f(x,y) \big) \\
&\ov{\eqref{Def-Kt-omega}}{=} \omega \mapsto  K\big((x,y) \mapsto\e^{\sqrt{2}\i (\W(1_{[0,t']} \ot (\alpha_x-\alpha_y)(\omega)}\e^{\sqrt{2}\i (\W(-1_{[t'+t,t']} \ot (\alpha_x-\alpha_y)))(\omega)}	f(x,y) \big)\\
&=\omega \mapsto K\big((x,y) \mapsto \e^{\sqrt{2}\i (\W(-1_{[t'+t,0]} \ot (\alpha_x-\alpha_y)(\omega)} f(x,y) \big) \quad \text{since } t+t' \leq 0 \leq t' \\
&\ov{\eqref{prem-series}}{=} \mathcal{U}_{t+t'}(1 \ot K_f) 
=U_{t+t'}(1 \ot K_f).
\end{align*} 
The remaining cases are left to the reader. With these computations, it is easy to check that $(U_t)_{t \in \R}$ is a weak* continuous group of $*$-automorphisms (consider the preadjoints maps of the $U_t$'s and use \cite[Lemme B.15]{EnN00} to prove the weak* continuity of the group). 

For any $f \in \L^2(X \times X)$, any $\xi \in \L^2(X)$, all $t \geq 0$ and almost all $y$ of $X$, we finally have\footnote{\thefootnote. Note that the function $f(\cdot,y)$ belongs to $\L^2(X)$ for almost all $y \in X$.}
\begin{align*}
\MoveEqLeft
 \big(\big[\E U_tJ(K_f)\big](\xi)\big)(y)           
		\ov{\eqref{Def-de-J-Schur}}{=} \big(\big[\E U_t(1 \ot K_f)\big](\xi)\big)(y)
		=\big(\big[\E \mathcal{U}_t S_t(1 \ot K_f)\big](\xi)\big)(y) \\
		&= \big(\big[\E \mathcal{U}_t(1 \ot K_f)\big](\xi)\big)(y) 
		=\bigg(\bigg[\int_{\Omega}  \big(\mathcal{U}_t(1 \ot K_f)\big)(\omega) \d \mu(\omega)\bigg](\xi)\bigg)(y) \\
		&=\int_{\Omega} \big(\big[ (\mathcal{U}_t(1 \ot K_f))(\omega)\big](\xi)\big)(y) \d \mu(\omega)\\
		&\ov{\eqref{prem-series}}{=} \int_{\Omega}\bigg(\int_{X} \e^{\sqrt{2}\i (\W_t(\alpha_x-\alpha_y))(\omega)} f(x,y)\xi(x) \d x\bigg) \d \mu(\omega)\\
		&=\int_{X}  \bigg(\int_{\Omega} \e^{\sqrt{2}\i (\W_t(\alpha_x-\alpha_y))(\omega)} \d \mu(\omega) \bigg) f(x,y)\xi(x) \d x
		\ov{\eqref{Esperance-exponentielle-complexe-2}}{=} \int_{X}  \e^{-t\norm{\alpha_x-\alpha_y}_{H}^2} f(x,y)\xi(x) \d x \\
		&\ov{\eqref{symbol-semigroup-Schur}}{=} \int_{X} \phi_t(x,y)f(x,y)\xi(x) \d x
		\ov{\eqref{Def-de-Kf}}{=}\big(K_{\phi_t f}(\xi)\big)(y).
\end{align*} 
Hence we have $\big(\E U_tJ(K_f)\big)(\xi)=\big(K_{\phi_t f}(\xi)$ and finally $\E U_tJ(K_f)=K_{\phi_t f}$. By weak* density, we deduce the existence of $T_t$ and \eqref{Dilation-Tt}. For any $t,t' \geq 0$ and all $x,y \in X$, we have 
$$
\e^{-t\norm{\alpha_x-\alpha_y}_{H}^2}\e^{-t'\norm{\alpha_x-\alpha_y}_{H}^2}
=\e^{-(t+t')\norm{\alpha_x-\alpha_y}_{H}^2}.
$$
We deduce that $(T_t)_{t \geq 0}$ is a semigroup. Furthermore, the function \eqref{symbol-semigroup-Schur} is real-valued. By Proposition \ref{Prop-selfadjoint}, we conclude that each $T_t$ is selfadjoint. The factorization \eqref{Dilation-Tt} shows that each $T_t$ is completely positive and unital. The weak* continuity is a consequence of the one of the group $(U_t)_{t \in \R}$.
\end{proof}


%

\begin{remark} \normalfont
It would be interesting to give a direct proof of the first part, i.e. without the factorization \eqref{Dilation-Tt}. 
\end{remark}

\section{Markov dilations of semigroups of measurable Schur multipliers}
\label{sec-Markov}



Our second main result is the following theorem which gives a standard Markov dilation.

\begin{thm}
\label{Th-Markov-dilation-Schur}
Let $X$ be a $\sigma$-finite measure space. Let $H$ be a real Hilbert space and $\alpha \co X \to H$ be a strongly measurable map. Consider the weak* continuous semigroup $(T_t)_{t \geq 0}$ of selfadjoint unital completely positive Schur multipliers on $\B(\L^2(X))$ defined in Theorem \ref{Th-dilation-continuous} by the symbols \eqref{symbol-semigroup-Schur}. There exists an injective von Neumann algebra $N$ equipped with a normal semifinite faithful trace, an increasing filtration $(N_s)_{s \geq 0}$ of $N$ with associated trace preserving conditional normal faithful expectations $\E_s \co N \to N_s$ and trace preserving unital normal injective $*$-homomorphisms $\pi_s \co \B(\L^2(X)) \to N_s$ such that
\begin{equation}
\label{Equa-standard-Markov}
\E_s\pi_t
=\pi_sT_{t-s}, \quad 0 \leq s \leq t.
\end{equation}
\end{thm}

\begin{proof}
Let $\W \co \L^2_\R(\R^+,H) \to \L^0(\Omega)$ be an $H$-cylindrical Brownian motion on a probability space $(\Omega,\mu)$, see Section \ref{sec:preliminaries}. As in the proof of Theorem \ref{Th-dilation-continuous}, for almost all $\omega \in \Omega$ and any $t \geq 0$, we introduce the element \eqref{Def-Kt-omega} and for any $t \geq 0$ the element $V_t$ of $\L^\infty(\Omega,\B(\L^2(X)))$ defined by \eqref{Def-Vt}. Recall that $V_t$ is an unitary element of the von Neumann algebra $N \ov{\mathrm{def}}{=} \L^\infty(\Omega,\B(\L^2(X)))$ which is clearly injective. Again with the $*$-homomorphism $J$ defined in \eqref{Def-de-J-Schur} and the $*$-automorphism $\mathcal{U}_t$ defined in \eqref{Def-Ut-Schur}, we define for any $t \geq 0$ the  unital normal injective $*$-homomorphism
\begin{equation}
\label{def-Pi-s-Schur-2}
\begin{array}{cccc}
\pi_t\ov{\mathrm{def}}{=} \mathcal{U}_tJ \co   &  \B(\L^2(X))   &  \longrightarrow   &  \L^\infty\big(\Omega,\B(\L^2(X))\big)  \\
        &  z   &  \longmapsto       &  V_t (1 \ot z)V_t^* \\
\end{array}.	
\end{equation}
It is obvious that $\pi_t$ is trace preserving. For any $t \geq 0$, we also define the canonical normal conditional expectations $\E_{\scr{F}_t} \co \L^\infty(\Omega) \to \L^\infty(\Omega,\scr{F}_t)$ and $\E_t \ov{\mathrm{def}}{=}\E_{\scr{F}_t} \ot \Id_{\B(\L^2(X))} \co \L^\infty(\Omega,\B(\L^2(X)) \to \L^\infty(\Omega,\scr{F}_t,\B(\L^2(X))$ where the $\sigma$-algebra $\scr{F}_t$ is defined in \eqref{Filtration-H-cylindrical}. We introduce the von Neumann algebra $N_t \ov{\mathrm{def}}{=} \L^\infty(\Omega,\scr{F}_t,\B(\L^2(X))$. Note that the range of $\pi_t$ is included in $N_t$.  

For almost all $\omega \in \Omega$, the Hilbert-Schmidt operator 
$$
\big(\pi_t(K_f)\big)(\omega)
\ov{\eqref{def-Pi-s-Schur-2}\eqref{Def-de-J-Schur}}{=} \big(\mathcal{U}_t(1 \ot K_f)\big)(\omega)
$$ 
is associated by the computation \eqref{prem-series-2} with the function 
\begin{equation}
\label{Equa-inter-5}
(x,y) \mapsto \e^{\sqrt{2}\i(\W_t(\alpha_x-\alpha_y))(\omega)} f(x,y).
\end{equation}
Similarly, for any $0 \leq s \leq t$ and almost all $\omega \in \Omega$, the operator
\begin{align*}
\MoveEqLeft
\big(\pi_s(T_{t-s}(K_f))\big)(\omega)           
=\big(\pi_s(K_{\phi_{t-s}f})	\big)(\omega) 
\end{align*} 
is Hilbert-Schmidt and associated with the function 
\begin{equation}
\label{Divers-32}
(x,y) \mapsto \e^{\sqrt{2}\i(\W_s(\alpha_x-\alpha_y))(\omega)} \e^{-(t-s)\norm{\alpha_x-\alpha_y}_{H}^2}f(x,y).
\end{equation}
Now recall that $(\W_t(h))_{t \geq 0}$ is a Brownian motion for any fixed $h \in H$. Hence for any $0 \leq s \leq t$ and any $x,y \in X$ the random variable 
\begin{align}
\MoveEqLeft
\label{Eq-inter-1}
\W\big(1_{]s,t]} \ot (\alpha_x-\alpha_y)\big)
=\W\big(1_{[0,t]} \ot (\alpha_x-\alpha_y)\big)-\W\big(1_{[0,s]} \ot (\alpha_x-\alpha_y)\big) \\
&\ov{\eqref{Def-Wt-h}}{=} \W_t(\alpha_x-\alpha_y)-\W_s(\alpha_x-\alpha_y) \nonumber
\end{align}
is independent by \eqref{difference-3} from the $\sigma$-algebra $\scr{F}_s \ov{\eqref{Filtration-H-cylindrical}}{=} \sigma\big(\W_r(h) : r \in [0,s], h \in H\big)$. Consequently, the random variable $\e^{\sqrt{2}\i\W(1_{]s,t]}\ot (\alpha_x-\alpha_y))}$ is also independent from the $\sigma$-algebra $\scr{F}_s$. Using \cite[Proposition 2.6.31]{HvNVW16} in the fourth equality since each random variable $\W_s(\alpha_x)$ is $\scr{F}_s$-measurable, we finally obtain for any $0 \leq s \leq t$ and any $f \in \L^2(X \times X)$
\begin{align*}
\MoveEqLeft
\E_s\pi_t(K_f)             
\ov{\eqref{Equa-inter-5}}{=} \E_s\Big(\omega \mapsto \big[K\big((x,y) \mapsto \e^{\sqrt{2}\i \W_t(\alpha_x-\alpha_y)(\omega)} f(x,y)\big)\big]\Big) \\
&\ov{\eqref{Eq-inter-1}}{=} \E_s\Big(\omega \mapsto \big[K\big((x,y) \mapsto\e^{\sqrt{2}\i\W_s(\alpha_x-\alpha_y)(\omega)}\e^{\sqrt{2}\i\W(1_{]s,t]}\ot (\alpha_x-\alpha_y))(\omega)}f(x,y)\big)\big]\Big)\\
&\ov{\eqref{formule-calcul-mult}}{=} \E_s\Big(\omega \mapsto \Mult_{k_{s,\omega}}\big[K\big((x,y) \mapsto \e^{\sqrt{2}\i\W(1_{]s,t]}\ot (\alpha_x-\alpha_y))(\omega)}f(x,y)\big)\big]\Mult_{\ovl{k_{s,\omega}}}\Big)\\
&=\big[\omega \mapsto \Mult_{k_{s,\omega}}\big] \E_s\big(\omega \mapsto \big[K\big((x,y) \mapsto \e^{\sqrt{2}\i\W(1_{]s,t]}\ot (\alpha_x-\alpha_y))(\omega)}f(x,y)\big)\big) \big[\omega \mapsto \Mult_{\ovl{k_{s,\omega}}} \big]\\
&\ov{\eqref{Hy-2.6.35}}{=}\big[\omega \mapsto \Mult_{k_{s,\omega}}\big] \E\big(\omega \mapsto \big[ K\big((x,y) \mapsto \e^{\sqrt{2}\i\W(1_{]s,t]}\ot (\alpha_x-\alpha_y))(\omega)}f(x,y)\big)\big]\big) \big[\omega \mapsto \Mult_{\ovl{k_{s,\omega}}} \big].
\end{align*} 
Now, for any $\xi \in \L^2(X)$ and almost all $y \in Y$, we see that
\begin{align*}
\MoveEqLeft
\E\big(\omega \mapsto \big[K\big((x,y) \mapsto \e^{\sqrt{2}\i\W(1_{]s,t]}\ot (\alpha_x-\alpha_y))(\omega)}f(x,y)\big)\big]\big) (\xi)(y)   \\   
&=\bigg(\int_\Omega K\big((x,y) \mapsto \e^{\sqrt{2}\i\W(1_{]s,t]}\ot (\alpha_x-\alpha_y))(\omega)}f(x,y)\big)\big) \d \mu(\omega)\bigg) (\xi)(y) \\     
&=\int_X \bigg(\int_\Omega \e^{\sqrt{2}\i\W(1_{]s,t]}\ot (\alpha_x-\alpha_y))(\omega)}\d \mu(\omega)\bigg)f(x,y)\xi(x) \d x \\
&\ov{\eqref{Esperance-exponentielle-complexe-2}}{=} \int_X \e^{-(t-s)\norm{\alpha_x-\alpha_y}_{H}^2}f(x,y)\xi(x) \d x.
\end{align*}
Consequently, the previous expectation is
\begin{align*}
\MoveEqLeft
\E\big(\omega \mapsto \big[K\big((x,y) \mapsto \e^{\sqrt{2}\i\W(1_{]s,t]}\ot (\alpha_x-\alpha_y))(\omega)}f(x,y)\big)\big]\big) 
=K\big((x,y) \mapsto \e^{-(t-s)\norm{\alpha_x-\alpha_y}_{H}^2}f(x,y)\big).
\end{align*}
Finally, we obtain
\begin{align*}
\MoveEqLeft
\E_s\pi_t(K_f)             
=\big[\omega \mapsto \Mult_{k_{s,\omega}}\big] K\big((x,y) \mapsto \e^{-(t-s)\norm{\alpha_x-\alpha_y}_{H}^2}f(x,y)\big) \big[\omega \mapsto \Mult_{\ovl{k_{s,\omega}}} \big]\\
&\ov{\eqref{formule-calcul-mult}}{=}  \omega \mapsto K\big((x,y) \mapsto \e^{\sqrt{2}\i\W_s(\alpha_x-\alpha_y)(\omega)}\e^{-(t-s)\norm{\alpha_x-\alpha_y}_{H}^2}f(x,y)\big)
\ov{\eqref{Divers-32}}{=} \pi_s(T_{t-s}(K_f)).
\end{align*}
By weak* density, the proof is complete.
\end{proof}


Similarly, we can prove the following reversed Markov dilation.

\begin{thm}
\label{Th-reversed-Markov-dilation-Schur}
Let $X$ be a $\sigma$-finite measure space. Let $H$ be a real Hilbert space and $\alpha \co X \to H$ be a strongly measurable map. Consider the weak* continuous semigroup $(T_t)_{t \geq 0}$ of selfadjoint unital completely positive Schur multipliers on $\B(\L^2(X))$ defined in Theorem \ref{Th-dilation-continuous} by the symbols \eqref{symbol-semigroup-Schur}. There exists an injective von Neumann algebra $N$ equipped with a normal semifinite faithful trace, a decreasing filtration $(N_s)_{s \geq 0}$ of $N$ with associated trace preserving normal faithful conditional expectations $\E_s \co N \to N_s$ and trace preserving unital normal injective $*$-homomorphisms $\pi_s \co \B(\L^2(X)) \to N_s$ such that
\begin{equation}
\label{-equa-reversed-Markov}
\E_s\pi_t
=\pi_sT_{s-t}, \quad 0 \leq t \leq s.
\end{equation}
\end{thm}

\section{Applications to functional calculus and interpolation}
\label{sec-Applications}

We refer to \cite{Haa06}, \cite{HvNVW18}, \cite{JMX06}, \cite{KuW04}, \cite{LeM07} for background on sectoriality and $\H^\infty$ functional calculus. In the spirit of the result of \cite[Theorem 4.1]{Arh20}, we can prove the following theorem with a similar agument.

\begin{thm}
\label{Th-funct-calculus}
Let $X$ be a $\sigma$-finite measure space. Let $H$ be a real Hilbert space and $\alpha \co X \to H$ be a strongly measurable map. Consider the weak* continuous semigroup $(T_t)_{t \geq 0}$ of selfadjoint unital completely positive Schur multipliers on $\B(\L^2(X))$ defined in Theorem \ref{Th-dilation-continuous} by the symbols \eqref{symbol-semigroup-Schur}. Suppose $1<p<\infty$. We let $-A_p$ be the generator of the induced strongly continuous semigroup $(T_{t,p})_{t \geq 0}$ on the Banach space $S^p_X$. Then for any $\theta>\pi|\frac{1}{p}-\frac{1}{2}|$, the operator $A_p$ has a completely bounded $\H^{\infty}(\Sigma_\theta)$ functional calculus.
\end{thm}

Let $\mathcal{T}=(T_t)_{t \geq 0}$ be a Markov semigroup on a von Neumann algebra $\mathcal{M}$ equipped with a normal semifinite faithful trace. We set $
\mathcal{M}_0 
\ov{\mathrm{def}}{=} \left\{  x \in \mathcal{M}  : \lim_{t \rightarrow \infty} T_t(x) = 0 \right\}$ where the limit is a weak* limit. For $1 \leq p < \infty$, we let $
\L^p_0(\mathcal{M})
\ov{\mathrm{def}}{=} \{x \in \L^p(\mathcal{M}) : \lim_{t \to \infty} T_t(x)=0 \}$ where the limit is for the $\L^p$-norm.

Recall that the column-BMO seminorm is defined by
$$
\norm{x}_{\BMO^c_{\mathcal{T}}} 
\ov{\mathrm{def}}{=} \sup_{t \geq 0}  \bnorm{  T_t\left(|x - T_t(x)|^2 \right) }_\infty^{\frac{1}{2}}, \quad x \in \L^2_0(\mathcal{M}).
$$
Then we can define the row-BMO seminorm $\norm{x}_{\BMO^r_{\mathcal{T}}} \ov{\mathrm{def}}{=} \norm{x^*}_{\BMO^c_{\mathcal{T}}}$ and finally the BMO-seminorm
$$	
\norm{x}_{\BMO_{\mathcal{T}}} 
\ov{\mathrm{def}}{=} \max\big\{ \norm{x}_{\BMO^c_{\mathcal{T}}}, \norm{x}_{\BMO^r_{\mathcal{T}}} \big\}.
$$  
In this context, we can introduce a Banach space $\BMO_{\mathcal{T}}$ and we refer to \cite[Section 3]{Mei08}, \cite{JM12} and \cite{Cas19} for a precise definition. Combinating Theorem \ref{Th-reversed-Markov-dilation-Schur} and \cite[Theorem 5.12]{JM12}, we obtain the following interpolation formula by oberving that each operator $T_t$ of the semigroup is trace-preserving by the factorization of Theorem \ref{Th-dilation-continuous}.

\begin{thm}
\label{Th-interpolation-BMO}
Let $X$ be a $\sigma$-finite measure space. Let $H$ be a real Hilbert space and $\alpha \co X \to H$ be a strongly measurable map. Consider the weak* continuous semigroup $(T_t)_{t \geq 0}$ of selfadjoint unital completely positive Schur multipliers on $\B(\L^2(X))$ defined in Theorem \ref{Th-dilation-continuous} by the symbols \eqref{symbol-semigroup-Schur}. Suppose $1\leq p<\infty$. Then 
$$
\big(\BMO_\mathcal{T},S^1_{X,0}\big)_{\frac{1}{p}}
\approx S^p_{X,0}
$$ 
with equivalence of norms up to a constant $\approx p$.
\end{thm}

\begin{remark} \normalfont
\label{Remarque-ultime}
We can replace the space $\BMO_{\mathcal{T}}$ by other $\BMO$-spaces, see \cite[Theorem 5.12]{JM12}.  
\end{remark}

Finally, we could give applications in ergodic theory in the spirit of the paper \cite{JiW17}.

\vspace{0.2cm}

\textbf{Acknowledgements}.
The author acknowledges support by the grant ANR-18-CE40-0021 (project HASCON) of the French National Research Agency ANR. I would like to thank Julio Delgado for an interesting  discussion. I am grateful to Lyudmyla Turowska for short discussions and the reference \cite{STT14} and Hermann K\"onig for a very instructive discussion on \cite[Theorem 3.a.1 p. 145]{Kon86}. 


\vspace{0.2cm}
\footnotesize{
\noindent C\'edric Arhancet\\ 
\noindent 6 rue Didier Daurat, 81000 Albi, France\\
URL: \href{http://sites.google.com/site/cedricarhancet}{https://sites.google.com/site/cedricarhancet}\\
cedric.arhancet@protonmail.com\\

\end{document}

\normalsize

\section{Discussion}

The following lemma is well-known.

\begin{lemma}
\label{Lemma-well-known}
Let $X$ be a locally compact space equipped with a Radon measure with full support. If $f,g \co X \times X \to \mathbb{C}$ are continuous with compact support then the function $h$ defined by \eqref{composition-Hilbert-Schmidt} is continuous with compact support.
\end{lemma}

Let $X$ be a second countable locally compact space equipped with a Radon measure with full support. Let $K_f \co \L^2(X) \to \L^2(X) $ be a \textit{positive} Hilbert-Schmidt operator with a \textit{continuous} kernel $f \co X \times X \to \mathbb{C}$. Then by \cite[Theorem 4.3]{Bri91} (see also \cite{Jef16} \cite{Del10} for related results) the operator $K_f$ is trace-class if and only if $\int_{X} f(x,x) \d x <\infty$. In this case, we have
\begin{equation}
\label{trace-of-trace-class}
\tr K_f
=\int_{X} f(x,x) \d x.
\end{equation}

\subsection{La formule du produit scalaire}
$$
\tr(K_f K_g)
=\int_{X \times X} f(x,y) g(y,x) \d x \d y.
$$
The selfadjointness property is equivalent to
$$
\tr\big(M_\varphi(K_f)K_g^*\big)
=\tr\big(K_f(M_\varphi(K_g))^*\big), \quad \text{i.e.} \quad \tr\big(K_{\varphi f}K_g^*\big)
=\tr\big(K_f(K_{\varphi g})^*\big)
$$
for any $f,g \in \L^2(X \times X)$. Since $\L^2(X \times X) \to S^2_X$, $f \mapsto K_f$ is an isometry from the Hilbert space $\L^2(X \times X)$ onto the Hilbert space $S^2_X$ of Hilbert-Schmidt operators on $\L^2(X)$, that means that
$$
\iint_{X \times X} \varphi(x,y)f(x,y)\ovl{g(x,y)} \d x\d y
=\iint_{X \times X} f(x,y) \ovl{\varphi(x,y)g(x,y)} \d x\d y.
$$

\subsection{Fukumizu }
Kenji Fukumizu (demander ref) en fait je crois que c'est fit dans Cartier

\newpage

\subsection{Convolution groupoids}

Now, we prove the converse. 
Suppose that $M_\varphi$ is trace preserving. Let $K_g$ be an element of $S^2_X$ with a continuous kernel $g \in \L^2(X \times X)$. Then if we define $h(x,y) \ov{\mathrm{def}}{=} \int_{X} \ovl{g(z,x)}g(z,y) \d z$ then by \eqref{composition-Hilbert-Schmidt} the composition $K_h=K_g^*K_g$ is a trace-class positive operator. Since $g$ is continuous, its kernel is continuous. Hence
\begin{align*}
\MoveEqLeft
\int_X \varphi(x,x) h(x,x) \d x
\ov{\eqref{trace-of-trace-class}}{=} \tr K_{\varphi h}
=\tr M_\varphi(K_h)            
=\tr K_h 
\ov{\eqref{trace-of-trace-class}}{=} \int_X h(x,x) \d x.
\end{align*}
That means that
$$
\int_{X} \varphi(x,x) \int_{X}\ovl{g(z,x)}g(z,x) \d z \d x
=\int_{X}\int_{X}\ovl{g(z,x)}g(z,x) \d z \d x.
$$
We obtain
$$
\int_{X} \varphi(x,x) (\check{\ovl{g}}*g)(x,x) \d x
=\int_{X} (\check{\ovl{g}}*g)(x,x) \d x.
$$
for the convolution on the groupoid.

\begin{prop}
\label{Prop-unital}
Let $X$ be a second countable compact space equipped with a Radon measure with support $X$. Let $M_\varphi \co \B(\L^2(X)) \to \B(\L^2(X))$ be a measurable Schur multiplier with a continuous symbol $\varphi \co X \times X \to \mathbb{C}$. If $M_\varphi$ is unital then for any $x \in X$ we have $\varphi(x,x)=1$.
\end{prop}

\begin{proof}
By Lemma \ref{Lemma-unital-1}, for any $k$ we have for the strong operator topology
$$
M_{\varphi_{\alpha}}(1_{\M_{n_{\alpha}}})
= \Psi_{\alpha} M_\varphi\Phi_{\alpha}(1_{\M_{n_{\alpha}}}) 
\xra[\alpha]{}
\Psi_{\alpha} M_\varphi(\Id_{\B(\L^2(X))})
=1_{\M_{n_{\alpha}}}.
$$
Recall that by Lemma \ref{Lemma-approx} the operator $M_{\varphi_{\alpha}} \co \M_{n_{\alpha}} \to \M_{n_{\alpha}}$ is the Schur multiplier on the matrix algebra $\M_{n_{\alpha}}$ associated with the matrix $\big[\frac{1}{\mu(A_{i})}\frac{1}{\mu(A_{j})}\int_{A_{i} \times A_{j}} \varphi \big]_{1 \leq i,j \leq n_{\alpha}}$. Now, it is clear that for any $1 \leq i \leq n_{\alpha}$ we have $\frac{1}{\mu(A_{i})^2} \int_{A_{i} \times A_{i}} \varphi \xra[\alpha]{} 1$. By the convergence martingale theorem \cite[Theorem 3.32]{HvNVW1}, we have the convergence
$$
\sum_{i,j=1}^{n_\alpha} \frac{1}{\mu(A_{i})} \frac{1}{\mu(A_{j})}\bigg(\int_{A_{i} \times A_{j}} \varphi \bigg) 1_{A_{i} \times A_{j}}
\ov{\eqref{Def-phi-alpha}}{=} \E_{\alpha}(\varphi) \to \varphi
$$ 
almost everywhere on $X \times X$. Note that if $x \in X_0$ we have (as \cite[(2.7) page 233]{Bri2})
$$
\E_{\alpha}(\varphi)(x,x)
=\frac{1}{\mu(A_{i_x})^2} \int_{A_{i_x} \times A_{i_x}} \varphi.
$$ 
The proof is complete.
\end{proof}

\subsection{Lusin filtration}
\begin{remark} \normalfont
\label{Remark-Lusin}
Let $(X,\cal{A},\mu)$ be a $\sigma$-finite measure space We could use the notion of Lusin $\mu$-filtration of \cite[Definition 2.2]{Jef1}, a filtration $(\mathcal{A}_k)_{k \in \mathbb{N}}$ for which there exists a conegligible subset $X_0$ and an increasing\footnote{\thefootnote. A partition $\alpha$ is finer than a partition $\beta$ if each element of $\alpha$ is a subset of some element of $\beta$.} sequence $(\alpha_k)$ of countable partitions of $X_0$ such that 
\begin{flalign}
& \label{isonormal-almost} \text{$\mathcal{A}_k$ is the $\sigma$-algebra generated by $\alpha_k$} \\
&\label{difference-1} \text{$\mathcal{A} \cap X_0=\vee_k \mathcal{A}_k$,} \\
&\label{difference-2} \text{for any $x \in X_0$ we have $0<\mu(U_k(x))<\infty$}.  
\end{flalign}
where $U_k(x)$ is the unique set of the partition $\alpha_k$ of $X$ which contains $x$.
Note that by \cite[Theorem 2.3]{Jef1} any $\sigma$-finite Borel measure $\mu$ on a Souslin space admits a Lusin $\mu$-filtration. See \cite[page 232]{Bri2} for related things.
\end{remark}

\subsection{Autres}

Let us recall the version of Mercer's theorem for $\sigma$-finite measure spaces stated in \cite[Theorem 31]{Car1}. Let $X$ be a $\sigma$-finite measure space. Let $\varphi \co X \times X \to \mathbb{C}$ be a measurable positive definite kernel such that $\int_X \varphi(x,x) \d x <\infty$. Then the map $\L^2(X) \to \L^2(X)$, $\xi \mapsto \int_{X} \varphi(\cdot,y)\xi(y) \d y$ is a well-defined trace-class positive operator. The eigenfunctions $f_n \in \L^2(X)$ of this operator associated with those $n$ such that $\lambda_n \not=0$ and normalized by $\norm{f_n}_2 = 1$ satisfy
\begin{equation}
\label{Mercer-Cartier}
\varphi(x,y) 
=\sum_{n=0}^{\infty} \lambda_n f_n(x) \ovl{f_n(y)}
\end{equation}
almost everywhere. Moreover, for any $x,y \in X$ the series $\sum_{n \geq 0} \lambda_n f_n(x) \ovl{f_n(y)}$ converges absolutely. Note that the assumption $\int_X \varphi(x,x) \d x$ of this theorem is satisfied if $X$ is finite and if $\varphi$ is bounded.

https://mathoverflow.net/questions/253678/the-topology-of-pointwise-convergence-with-the-adjoint-operator-on-a-von-neumann
https://mathoverflow.net/questions/258829/operator-topologies-on-l-inftyx-mu

\vspace{0.2cm}

-preuve Lemma \ref{Lemma-iso-1}.

-th avec les angles Noncommutative harmonic analysis on semigroups. Yong Jiao,

-resultats ergodic consequence de l'angle

-cas d'un mult de Schur et d'un mult de Fourier (i.e. sans semigroupe)

-preuve de la carac de Schur

-mult de Birman

-open the dorr to transference

-continuite des path

\subsection{Konig}

Let us state a \textit{corrected} version of Mercer's theorem for finite measure spaces stated in \cite[Theorem 3.a.1 page 145]{Kon1}. Note that the original version (and its proof) contains some problematic errors\footnote{\thefootnote. In particular, the estimate $\sup_n \norm{f_n}_\infty <\infty$ is note proved and seems incorrect and consequently the arguments of the convergence of  \cite[(3.3)]{Kon1} and of $\sum_{n \in \N} \lambda_n(T_k) (f,f_n) f_n$ are inexact.}. However, the statement can be corrected as follows using the \textit{same} nice ideas. The only contribution is to provide a correct statement and proof.

\begin{thm}
\label{Mercer-corrected}
Let $X$ be a finite measure space and $\varphi \in \L^\infty(X \times X)$ such that the Hilbert-Schmidt operator $K_\varphi \co \L^2(X) \to \L^2(X)$ is positive. 
The eigenfunctions $f_n \in \L^2(X)$ of the operator $K_{\varphi}$ associated with those $n$ such that $\lambda_n \not=0$, and normalized by $\norm{f_n}_2 = 1$, actually belong to $\L^\infty(X)$ with  
\begin{equation}
\label{Mercer}
\sup_{n \geq 0} \lambda_n \norm{f_n}_\infty^2 < \infty
\quad \text{and} \quad
\varphi(x,y) 
=\sum_{n=0}^{\infty} \lambda_n f_n(x) \ovl{f_n(y)}
\end{equation}
holds almost everywhere where the series converges absolutely almost everywhere and in $\L^\infty(X \times X)$. 
\end{thm}

\begin{proof}
The positive map $K_\varphi \co \L^2(X) \to \L^2(X)$ has a unique positive square root $S \co \L^2(X) \to \L^2(X)$. For any $\xi \in \L^1(X)$ we have
\begin{align*}
\MoveEqLeft
\norm{S(\xi)}^2_2
=\la S(\xi), S(\xi)\ra_{\L^2(X)}            
=\big\la S^2(\xi), \xi \big\ra_{\L^2(X)} 
=\la K_\varphi(\xi), \xi\ra_{\L^2(X)} \\
&=\iint_{X \times X} \varphi(x,y) \ovl{\xi(x)} \xi(y) \d x \d y
\leq \norm{\varphi}_{\L^\infty(X \times X)} \norm{\xi}_{\L^1(X)}^2.
\end{align*}
This proves that we have a well-defined continuous operator $S \co \L^1(X) \to \L^2(X)$. By duality, we have a bounded  operator $S \co \L^2(X) \to \L^\infty(X)$ with $
\norm{S}_{\L^2(X) \to \L^\infty(X)} 
\leq \norm{\varphi}_{\L^\infty(X \times X)}^{\frac{1}{2}}$. For any integer $n$, note that
$$
\lambda_n \norm{f_n}_\infty^2
=\bnorm{\sqrt{\lambda_n}f_n}_\infty^2
=\norm{S(f_n)}_\infty^2 
\leq \norm{S}_{\L^2(X) \to \L^\infty(X)}^2. 
$$
Hence, we deduce the estimate of \eqref{Mercer}. Recall that 
the embedding $i \co \L^\infty(X) \mapsto \L^2(X)$ is 2-summing with $\norm{i}_{\pi_2, \L^\infty(X) \mapsto \L^2(X)} =\mu(X)^{\frac{1}{2}}$. Consequently, the operator $S \co \L^2(X) \to \L^2(X)$ is 2-summing by composition with 
\begin{equation}
\label{estimate-pi_2}
\norm{S}_{\pi_2, \L^2(X) \to \L^2(X)} 
\leq \mu(X)^{\frac{1}{2}} \norm{\varphi}_{\L^\infty(X \times X)}^{\frac{1}{2}}
\end{equation}
Now we have using \cite[Proposition 2.a.1 page 79]{Kon1} in the first inequality
\begin{align*}
\MoveEqLeft
\sum_{n=0}^\infty \lambda_n
=\sum _{n=0}^\infty \lambda_n(S)^2            
\leq \norm{S}_{\pi_2, \L^2(X) \to \L^2(X)} 
\ov{\eqref{estimate-pi_2}}{=} \mu(X) \norm{\varphi}_{\L^\infty(X \times X)}.
\end{align*}
Hence $K_\varphi$ is trace-class. Note that $K_\varphi(\xi)=\sum_n \lambda_n\la \xi, f_n\ra f_n$ for any $\xi \in \L^2(X)$. This series converges absolutely almost everywhere. The end of the proof is classical. 
\end{proof}

See the introduction of \cite{StS1} for a nice discussion of different versions of Mercer's theorem. Note also the version of Mercer's theorem stated \cite[Theorem 31]{Car1}.

\begin{proof}
Now, we prove to the second implication. Suppose that $\varphi$ is an integrally positive definite kernel. For any $\xi \in \L^2(X)$, we have the inequality \eqref{Cartier-int-pos}, i.e. $\la K_{\varphi}(\xi), \xi \ra_{\L^2(X)} \geq 0$. We infer that the operator $K_\varphi$ is positive. 
Using the above version of Mercer's theorem, we can define for any integer $n \geq 0$ the element $\alpha_n \ov{\mathrm{def}}{=} \sqrt{\lambda_n}f_n$ of $\L^2(X)$. For almost all $x \in X$, we have
$$
\sum_{n=0}^{\infty} |\alpha_n(x)|^2
=\sum_{n=0}^{\infty} \lambda_n|f_n(x)|^2
\leq \sum_{n=0}^{\infty} \lambda_n \norm{f_n}^2_\infty
<\infty.
$$
So we have an almost everywhere defined function $\alpha \co X \to \ell^2$, $x \mapsto (\alpha_n(x))_{n \geq 0}$. From \eqref{Mercer}, it is clear that 
$$
\varphi(x,y)
=\sum_{n=0}^{\infty}  \sqrt{\lambda_n}f_n(x)\ovl{\sqrt{\lambda_n}f_n(y)}
=\la \alpha(x),\alpha(y) \ra_{\ell^2}
$$ 
for almost all $x,y \in X$. Extending $\alpha$ on the whole $X$, we conclude with Example \ref{Example-kernel-2}. The remaining assertions are easy and left to the reader.
\end{proof}

\begin{proof}
The Hilbert-Schmidt operator $K_\varphi \co \L^2(X) \to \L^2(X)$ is compact and positive. By the spectral theorem, we can write
\begin{equation}
\label{Spectral-decomposition}
K_\varphi(\xi)
=\sum_{i=0}^\infty \lambda_i \la u_i,\xi \ra u_i, \quad \xi \in \L^2(X)
\end{equation}
with $\lambda_i \geq 0$ and $u_i \in \L^2(X)$. For any $i \geq 0$, define $\alpha_i =\sqrt{\lambda_i}u_i$ and the element $\alpha \co X \to \ell^2$, $x \mapsto (\alpha_i(x))$ in \textbf{?$\ell^2$?}. For any $\xi \in \L^2(X)$ and any $y$, we have  
\begin{align*}
\MoveEqLeft
\int_X \la \alpha(x), \alpha(y) \ra_{\ell^2} \xi(x) \d x
=\int_X \sum_{i=0}^{\infty} \ovl{\alpha_i(x)}\alpha_i(y) \xi(x) \d x
=\int_X \sum_{i=0}^{\infty} \lambda_i \ovl{u_i(x)}u_i(y)\xi(x) \d x  \\         
&?=?\sum_{i=0}^{\infty} \lambda_i u_i(y) \int_X \ovl{u_i(x)} \xi(x) \d x 
=\sum_{i=0}^\infty \lambda_i \la u_i,\xi \ra u_i(y) 
\ov{\eqref{Spectral-decomposition}}{=} (K_\varphi \xi)(y)
\ov{\eqref{Def-de-Kf}}{=} \int_{X} \varphi(x,y) \xi(y) \d x.
\end{align*}
Hence $
\la \alpha(x),\alpha(y) \ra_{\ell^2}
=\varphi(x,y)
$ for almost all $x,y \in X$. We conclude with Example \ref{Example-kernel-2}.
\end{proof}

\begin{remark} \normalfont
\label{remark-open-question}
It would be very useful to know if we can remove the assumption ``which induces a bounded Schur multiplier'' in the above result.
\end{remark}

The following is a partial generalization of \cite[Proposition 5.4]{Arh1}. 

	%
	%
%

\begin{proof}
Observe that $\varphi$ belongs to $\L^2(X \times X)$. So we have a well-defined Hilbert-Schmidt operator $K_\varphi \co \L^2(X) \to \L^2(X)$. If $\xi$ belongs to the Hilbert space $\L^2(X)$, note that $\ovl{\xi} \ot \xi$ belongs to $\L^2(X \times X)$. Hence the function $(x,y) \mapsto \varphi(x,y) \ovl{\xi(x)} \xi(y)$ is integrable on $X \times X$ and by Fubini's theorem we have 
\begin{equation}
\label{Fubini}
\la K_{\varphi}(\xi), \xi \ra_{\L^2(X)}
\ov{\eqref{Def-de-Kf}}{=} \iint_{X \times X} \varphi(x,y) \ovl{\xi(x)} \xi(y) \d x \d y .
\end{equation} 
We conclude with the above Cartier's observation.
\end{proof}

A generalization of the above result to $\sigma$-finite measures spaces would also be welcome, maybe using an elementary argument to reduce the $\sigma$-finite case to the finite case. It would be interesting to know if the product of two integrally positive
definite kernels of $\L^\infty(X)$ is again an integrally positive definite kernel.

\begin{remark} \normalfont
\label{Rem-int-pos-continuous}
Let $X$ be a locally compact topological space and $\mu$ be a Radon measure on $X$ with support $X$. Consider a \textit{continuous} function $\varphi \co X \times X \to \mathbb{C}$. Then the author believe that it is folklore that $\varphi$ is an integrally positive definite kernel if and only if $\varphi$ is a positive definite kernel on $X \times X$.
\end{remark}

\end{document}